\documentclass[12pt,a4paper]{article}

\usepackage[nottoc]{tocbibind}
\usepackage{mathtools}
\usepackage{preamble}

\newcommand{\usualHilb}{{\ccH}ilb}
\newcommand{\usualHilbsm}[2][r]{\usualHilb_{#1} ^{\operatorname{sm}}{(#2)}}
\DeclareMathOperator{\sat}{sat}
\DeclareMathOperator{\sch}{\mathsf {sch}}
\DeclareMathOperator{\spanmap}{\mathsf {span}}
\DeclareMathOperator{\idealmap}{\mathsf {ideal}}
\DeclareMathOperator{\perpmap}{\mathsf {perp}}

 \usepackage{vmargin}
 \setmargrb{1in}{0.7in}{1in}{0.7in} 

\title{Apolarity for border cactus decompositions}
\author{\nisiabu \and \JaBu}

\date{27th January 2026}

\newenvironment{red}{\color{red}}{}
\newcommand{\bred}{\begin{red}}
\newcommand{\ered}{\end{red}}

\newenvironment{blue}{\color{blue}}{}
\newcommand{\bblue}{\begin{blue}}
\newcommand{\eblue}{\end{blue}}

\newenvironment{green}{\color{green}}{}
\newcommand{\bgreen}{\begin{green}}
\newcommand{\egreen}{\end{green}}

\usepackage[textsize=tiny]{todonotes}

\newcommand{\bCD}{\underline{\mathbf{CD}}}
\newcommand{\Irrel}[1]{\operatorname {Irrel}_{#1}}

\newcommand{\apolar}[1][F]{\operatorname{Ann}\left(#1\right)}

\DeclareMathOperator{\Slip}{Slip}
\DeclareMathOperator{\Fl}{Fl}

\newcommand{\relGr}{\ccG{}r}

\begin{document}
\maketitle
\begin{abstract}
   The border apolarity technique was introduced in our earlier work for secant varieties over complex numbers.
   We extend the theory to cactus varieties of toric varieties over any algebraically closed field.
   A border cactus decomposition is a mulithomogeneous ideal in the Cox ring (also called the total coordinate ring) of the toric variety that witnesses that a given point is in a specific cactus variety.
   The definition of such witness uses apolarity and
   we describe the set of ideals that are credible witnesses for this purpose in terms of a correspondence between the usual Hilbert scheme (parametrising all closed subschemes of the toric variety) and the multigraded Hilbert scheme (parametrising all multihomogeneous ideals in the Cox ring).
   We also take this opportunity to extend the border apolarity to linear subspaces (in non-border setting, this is equivalent to simultaneous decompositions).
\end{abstract}

\medskip
{\footnotesize
\noindent\textbf{addresses:} \\
W.~Buczy\'nska, \eemail{wkrych@mimuw.edu.pl},
   Faculty of Mathematics, Computer Science and Mechanics, University of Warsaw, ul. Banacha 2, 02-097 Warsaw, Poland\\
J.~Buczy\'nski, \eemail{jabuczyn@impan.pl},
   Institute of Mathematics of the Polish Academy of Sciences, ul. \'Sniadeckich 8, 00-656 Warsaw, Poland\\
}

\noindent\textbf{keywords:}
cactus varieties, border cactus rank, finite schemes, apolarity, multigraded Hilbert scheme,  Cox ring.


\tableofcontents

\section{Introduction}

Throughout the article $\kk$ is
an algebraically closed base field of any characteristics, and $\NN$ denotes the set of non-negative integers.
Fix a projective space $\PP W \simeq \PP^N$ over $\kk$.
Given a subscheme $R\subset \PP^N$ by $\linspan{R}$ we denote the projective linear span of $R$, that is the linear space $\PP^k\subset \PP^N$
defined by all the degree $1$ polynomials in the homogeneous ideal of $R$.
For a projective variety $X \subset \PP^N$ the $r$-th \emph{secant variety} of $X$
  is the closure of the union of all linear subspaces spanned by $r$ points of $X$:
\[
  \sigma_r(X) := \bigcup \overline{\set{\linspan{\fromto{p_1}{p_r}}  \mid p_i \in X}} \subset \PP^N.
\]
It is a natural construction of a sequence of varieties starting from $X=\sigma_1(X)$ and
--- if $X$ is not contained in any hyperplane ---  stabilising at $\sigma_r(X)=\PP^N$ for $r\gg0$.
Of main interest are in particular the cases where $X\subset \PP W$ is either a Segre variety consisting of simple tensors in the tensor space $\PP (A\otimes B \otimes C \otimes\dotsb)$, Veronese variety consisting of divided powers of linear forms in the space $\PP(S^{(d)} A)$ of homogeneous polynomials of degree $d$, or the Segre-Veronese variety, which is a combination of both above.
In \cite{nisiabu_jabu_border_apolarity}, for those varieties,
and more generally for smooth toric projective varieties,
we proposed a method of studying secant varieties called border apolarity, see \cite[Thms~1.2, 3.15, 4.3]{nisiabu_jabu_border_apolarity}.
The goal of this article  is to improve and  expand the method of border apolarity.
In particular, we address the problems posed in
\cite[\S7.2, \S7.4, \S7.5]{nisiabu_jabu_border_apolarity}.

\subsection{Motivation and historical remarks about secants}

Secants of plane curves (for example, of a circle) have been present in geometry since antiquity.
Versions of secant varieties were studied as early as the first half of the 20th century.
Whitney proved that any smooth  manifold $M$ of dimension $n$ (over $\RR$) can be embedded smoothly in $\RR^{2n+1}$ using projections from a point that lies outside of the differential-geometric version of the second secant variety \cite[\S3]{persson_Whitney_embedding_thm}.
By a more complicated method, he then showed that it is possible to project once more and obtain an embedding in $\RR^{2n}$.
Within broadly understood algebraic geometry, Sylvester, Palatini, Terracini, studied the properties of secant varieties and a related notion of $X$-rank (in special cases and not under that name). Later (around 1970), in one of the periods of intensive development of the theory of computational complexity, $X$-ranks (above all, for $X$ being the Segre variety of three factors) appeared while studying fast algorithms for multiplying large matrices. At the beginning of the 21st century, geometry and algorithms ``met'' and researchers in these two fields began to exchange knowledge intensively. It turned out that some important ``open'' problems for one party were well known and understood by the other, while the geometric understanding of aspects of the algorithms helped to generalise these algorithms widely,
or to prove they cannot be strenghtened due to natural geometric obstructions.
Since 2010 some of these obstructions have a name, they are called cactus varieties, see \S\ref{sec_intro_cacti}.

Consequently, the motivation for studying secant varieties and related notions of $X$-rank, $X$-border rank, and cactus varieties comes from their rich history, from
relatively elementary definitions combined with highly non-trivial problems,
from a natural geometric formulation,
but also from numerous applications to practical problems.
You can read more about this for instance in
\cite[Chapt.~1]{landsberg_tensorbook},
\cite[Chapt.~1]{landsberg_geometry_and_complexity}.

\subsection{Rank, decompositions, and cactus varieties}\label{sec_intro_cacti}

With the same assumptions about $X$ as above the $r$-th \emph{cactus variety} of $X$ is:
\[
  \cactus{r}{X}:= \bigcup \overline{\set{\langle R \rangle  \mid R \text{ is a $0$-dimensional  subscheme of $X$ degree at most $r$}  }} \subset \PP^N.
\]
Just as secant varieties, also cactus varieties form a natural sequence of reduced schemes starting
from $X=\cactus{1}{X}$ and  stabilising at $\cactus{r'}{X}=\PP^N$ for $r'\gg0$
--- again, subject to the assumption that $X$ is not contained in any hyperplane.
By construction $\sigma_r(X)\subset \cactus{r}{X}$ and if $X$ is smooth, then for some
small values of $r$ we have $\sigma_r(X)=  \cactus{r}{X}$.
However, for $X$ of sufficiently large dimension, the cactus varieties tend to fill out
the ambient projective space $\PP^N$ much faster than
the secant varieties~\cite[Thms~3, 4]{bernardi_ranestad_cactus_rank_of_cubics}, \cite[Cor.~1.7]{galazka_mgr_publ}.
Moreover, the cactus varieties form an obstruction for detereminantal equations
to define secant varieties \cite{nisiabu_jabu_cactus}, \cite{galazka_vb_cactus}, \cite[Thm~1.18]{galazka_phd}.

The notions of secant and cactus varieties lead to various types of ranks and decompositions of points $F\in \PP^N$:

\begin{itemize}
 \item  \emph{$X$-rank} of $F$ is $r_X(F)
      = \min\set {r \mid
        F\in \linspan{\fromto{p_1}{p_r}},
          p_i \in X}$.
 \item \emph{$X$-cactus rank} of $F$ is
      \[
        cr_X(F) = \min\set {\deg R \mid
        R \subset X \text{ is a subscheme}, \dim R =0,  F\in\linspan{R}}.
      \]
 \item A \emph{(minimal) decomposition} of $F$
   is any set $\setfromto{p_1}{p_{r_X(F)}}$
   with $p_i \in X$ such that $F\in \linspan{\fromto{p_1}{p_{r_X(F)}}}$.
 \item A \emph{cactus decomposition} of $F$
   is any subscheme $R\subset X$ such that $\dim R=0$, $\deg R = cr_X(F)$ such that $F\in \linspan{R}$.
\end{itemize}

\subsection{Apolarity: towards border decompositions}
In this article we are primarily interested in border analogues of the above notions of ranks and decompositions.
To define them precisely we need to recall the concept of apolarity.
Suppose
$S=\kk[\fromto{\alpha_1}{\alpha_m}]$ be a polynomial ring.
Consider $\widetilde{S}=\kk[\fromto{x_1}{x_m}]$ the vector space of polynomials in another set of variables $x_i$.
For now we ignore the multiplicative structure of $\widetilde{S}$ and we work only with its vector space structure.
For $I= (\fromto{i_1}{i_m}) \in \NN^{m}$ we denote by $\alpha^I$ the monomial
$\alpha_1^{i_1}\alpha_2^{i_2}\dotsm\alpha_m^{i_m}$ and by $x^{(I)}:=x_1^{i_1} x_2^{i_2}\dotsm x_m^{i_m}$.
We have an action $\hook$ of $S$ on $\widetilde{S}$ defined in coordinates as
\[
 \alpha^I \hook x^{(J)} =
 \begin{cases}
    x^{(J-I)} &\text{if } J-I \in \NN^m\\
    0 &\text{otherwise,}
 \end{cases}
\]
and we  extend  it bilinearly to a map  $S\times \widetilde{S}\to \widetilde{S}$.
This action is called the \emph{apolarity action}  and we discuss it  in Section~\ref{sec_apolarity}.

For any $F\in \widetilde{S}$ we define the apolar ideal $\apolar \subset S$:
\[
 \apolar = \set{\Theta\in S \mid \Theta\hook F = 0}.
\]
Note that for any non-zero constant $c\in \kk$ we have $\apolar[c\cdot F] = \apolar$.
Moreover, if $S$ is a graded ring by some finitely generated abelian group and $\tilde{S}$ is graded in the same way (that is $\deg \alpha_i = \deg x_i$ for each $i$), then for a homogeneous $F$ the ideal $\apolar$ is homogeneous.
Similar property applies to appropriate group actions: whenever the actions on $S$ and $\tilde{S}$ are compatibile, and $F$ is invariant, then also $\apolar$ is invariant.

Suppose now $X\subset \PPof{H^0(X,L)^*}$ is a smooth toric projective variety embedded via complete linear system of a very ample line bundle $L$.
The cases of major interest include Segre, Veronese and Segre--Veronose varieties, that is $X=\PP^{a_1}\times\dotsb \times \PP^{a_k}$   and $L=\ccO_X (\fromto{d_1}{d_k})$ for some positive integers $k$, $a_i$, and $d_i$
(Segre case is when all $d_i=1$, and Veronese case is when $k=1$).
These cases are connected to applications to tensors, as explained in the introduction to \cite{nisiabu_jabu_border_apolarity}.

Let $S= S[X]$ be the Cox ring of $X$,
so that $S=\bigoplus_{D\in \Pic(X)} H^0(X,D) \simeq \kk[\fromto{\alpha_1}{\alpha_m}]$, where $\Pic(X)\simeq \ZZ^{\rho}$ is the grading group and $m= \dim X + \rho$.
Then using the duality determined by apolarity we also have
$\widetilde{S}  = \bigoplus_{D\in \Pic(X)} H^0(X,D)^*$ and for any $F\in \PPof{H^0(X,L)^*}$ the apolar ideal $\apolar\subset S$ is a homogeneous ideal with respect to the grading determined by $\Pic(X)$.
Moreover, the automorphism group of $X$ acts on $S$, $\widetilde{S}$, and $\PPof{H^0(X,L)^*}$ in a compatibile way, and if $F\in \PPof{H^0(X,L)^*}$ is invariant under some subgroup $B\subset \Aut(X)$, then also $\apolar\subset S$ is preserved under the action of $B$ on $S$.

For each subscheme $R\subset X$ we define
its homogeneous ideal $I(R)\subset S$ to be generated by sections that vanish on $R$:
\[
  I(R):=\left(
    s\in S_{D}\simeq H^0(X,D) \mid D\in \Pic(X), s|_{R} = 0 \in H^0(R,D|_{R})
  \right).
\]
Ideals of the form $I(R)$ are \emph{saturated}.
Note that $\apolar$ is never saturated.
Then the apolarity theory
(in particular, Proposition~\ref{prop_multigraded_apolarity})
implies the following alternative definitions of decompositions and ranks, that shift attention from schemes or finite collections of points to ideals in $S$. Fix $F\in \PPof{H^0(X,L)^*}$ and a subscheme $R\subset X$. We have:

\begin{itemize}
 \item $F\in \linspan{R}$ if and only if $I(R)\subset \apolar$.
 \item
 $\setfromto{p_1}{p_{r(F)}}$
   is a decomposition if and only if
   $I(\fromto{p_1}{p_{r(F)}})\subset \apolar$.
 \item
 $R$ is a cactus decomposition
   if and only if
   $I(R) \subset \apolar$ and
   for all $D\in \Pic(X)$ we have
   $\dim_{\kk} (S/I)_D \le r$.
 \item  The $X$-rank of $F$ is equal to
 \[
   \min\set {r \mid
      \exists I\subset S \text{ homog., radical, saturated, } \forall_D \dim_{\kk} (S/I)_D \le r \text{, and }  I \subset \apolar}.
 \]
 \item The $X$-cactus rank of $F$ is equal to
 \[
   \min\set {r \mid
      \exists I\subset S \text{ homogeneous, saturated, } \forall_D \dim_{\kk} (S/I)_D \le r \text{, and }  I \subset \apolar}.
 \]
\end{itemize}

\subsection{Main results}

With these ideas in mind, we define the border decomposition of $F$ to be an appropriate ideal of $S$ contained in  $\apolar$ and witnessing the membership of $F\in \sigma_r(X)$.
Our first result is the generalisation of border apolarity \cite[Thm~1.2]{nisiabu_jabu_border_apolarity} to fields of any characteristics.
\begin{thm}[Weak border apolarity]
    \label{thm_border_apolarity_intro}
    Suppose $F\in \sigma_r(X)$.
    Then there exists a homogeneous ideal $I \subset S$ such that:
    \begin{itemize}
     \item $I\subset \apolar$,
     \item for each multidegree $D$ the $D$-th graded piece $I_D$ of $I$
           has codimension (in $S[X]_D$) equal to $\min(r, \dim S[X]_D)$.
    \end{itemize}
    Moreover, if $B$ is a connected solvable group acting on $X$ and preserving $F$,
      then there exists an $I$ as above which in addition is invariant under
      $B$.
\end{thm}
The main nontrivial difference between this statement and the arguments in \cite{nisiabu_jabu_border_apolarity}
is in a strengthening of \cite[Lem.~3.9]{nisiabu_jabu_border_apolarity} about a very general tuple of points to an analogous claim for a general tuple of points.
This is dealt with in Theorem~\ref{thm_Sip_open}.
The second result generalises the border decompositions above to border cactus decomposition.

\begin{thm}[Weak Apolarity for Border Cactus Decompositions, or weak ABCD]\label{thm_weak_abcd}
   Fix a toric variety $X$ as above and a positive integer $r$.
   Then there exists a finite collection $\setfromto{h_1}{h_k}$ of Hilbert functions
   $h_i\colon \Pic(X)\to \NN$
     that have the following properties.
   Firstly, they are Hilbert functions of a finite subscheme $R_i\subset X$ of degree $r$, that is $h_i(D)= \dim_{\kk}(S/I(R_i))_D$.
   Secondly, suppose $L$ is a very ample line bundle and
     $F \in \PP(H^0(X,L)^*)$.
     If  $F \in \cactus{r}{X}$ (with $X$ embedded using the linear system $|L|$)
     then there exists a homogeneous ideal
     $I \subset S[X]$
     with Hilbert function~$h_i$ for some $i$,
   such that $I \subset \apolar$.
\end{thm}

The above theorem is a consequence of stronger Theorem~\ref{thm_abcd},
   which is an ``if and only if'' statement.
For this purpose, we compare the multigraded Hilbert scheme
$\Hilb(S[X])$ and the usual Hilbert scheme $\usualHilb(X)$.

\begin{prop}\label{prop_distinguished_components}
   For each irreducible component
      $\ccH\subset \reduced{\usualHilb(X)}$
      there is a unique irreducible component $\mathbf{H}_{\ccH} \subset \reduced{\Hilb(S[X])}$ such that:
      \begin{itemize}
       \item a general ideal $I\in \mathbf{H}_{\ccH}$ is saturated,
       \item the natural map $\sch\colon\Hilb(S[X]) \to \usualHilb(X)$
             taking a homogeneous ideal $I\subset S[X]$
             to the subscheme of $X$ defined by $I$
             restricts to a birational morphism
              $\mathbf{H}_{\ccH}\to \ccH.$
      \end{itemize}
\end{prop}

For each $\ccH$ as in the proposition, there is an integer valued function
$h_{\ccH}\colon \Pic(X)\to \NN$
which is the Hilbert function of any $I\in \mathbf{H}_{\ccH}$.
Equivalently,
  if $[R]\in \ccH$ is a general element, and $R\subset X$ is the corresponding scheme, then
  $h_{\ccH}$ is the Hilbert function of $I_R\subset S[X]$, see Theorem~\ref{thm_structure_of_Hilbert_functions_of_fibres}.

\begin{thm}[Apolarity for Border Cactus Decompositions, or ABCD]\label{thm_abcd}
   Suppose $F \in \PPof{H^0(X,L)^*}$.
   Then $F\in \cactus{r}{X}$ if and only if for some irreducible component
   $\ccH\subset \usualHilb_r(X)$ of the Hilbert scheme of points
     of length $r$ there exists a homogeneous ideal $I \subset S[X]$
     with $[I]\in \mathbf{H}_{\ccH}$ and $I \subset \Ann(F)$.
   Moreover, $\ccH$ might be chosen to be a component containing
     Gorenstein schemes.
\end{thm}

Theorems~\ref{thm_border_apolarity_intro}, \ref{thm_weak_abcd}, \ref{thm_abcd} are stated and proven more generally, when instead of $F\in \PPof{H^0(X,L)^*}$ we have a linear subspace $W\subset H^0(X,L)^*$ of any dimension (not necessarily equal to $1$).

\subsection*{Acknowledgements}
We are grateful to Joachim Jelisiejew, Joseph Landsberg, and Tomasz Mańdziuk
  for their helpful comments.
The authors are supported by the research grant ``Complex contact manifolds and geometry of secants'', 2017/26/E/ST1/00231 awarded by National Science Center, Poland.
In addition, J.~Buczy{\'n}ski is partially supported by the project ``Advanced problems in contact manifolds, secant varieties, and their generalisations (\mbox{APRICOTS+})'',  2023/51/B/ST1/02799, also awarded by National Science Center, Poland.
Moreover, part of the research towards the results of this article was done during the scientific semester Algebraic Geometry with Applications to Tensors and Secants in Warsaw.
The authors are greatful to many fruitful discussions with the participants, and for the partial support by the Thematic Research Programme ``Tensors: geometry, complexity and quantum entanglement'', University of Warsaw, Excellence Initiative---Research University and the Simons Foundation Award No.~663281 granted to the Institute of Mathematics of the Polish Academy of Sciences for the years 2021-2023.

\section{Multihomogeneous ideals, saturation, and toric varieties}
In this section,
 we recall a result that under minor homogeneity assumptions the set of saturated ideals in a flat family of multihomogeneous ideals is open.
Moreover, we show that a general configuration of points in a smooth complete toric variety has the expected multigraded Hilbert function.

For the purpose of recursive self referencing we work here over an arbitrary base field
$\KK$ of arbitrary characteristics,  without assuming it being algebraically closed.
This choice  allows us to work  with families of ideals paramaterised by various schemes, for instance $\Spec \kk[t]$.
We need our arguments to be applicable also to non-closed fibres of our family, for instance, to ideals
in $\kk(t)[\fromto{\alpha_1}{\alpha_m}]$. In this case  $\KK=\kk(t)$, which is never algebraically closed.

\subsection{Openness of saturation and generic fibre}

We work with a polynomial ring $S=S_{\KK}= \KK[\fromto{\alpha_1}{\alpha_m}]$.
Suppose $A$ is a finitely generated abelian group,
that is $A\simeq \ZZ^q \oplus A^{\tor}$, where $q\geqslant 0$, $A^{\tor}$ is a finite abelian group.
An $A$-grading of $S$ is a homomorphism (called the \emph{degree}) from the semigroup of monomials in $S_{\KK}$ to $A$.
An $A$-grading is \emph{positive} if the only monomial of degree $0$ is $1\in S_{\KK}$.
Note that the grading is positive if and only if a
$S_D$ is finite dimensional vector space over $\KK$ for all $D\in A$.
In the situations considered in this article we  sometimes vary the field $\KK$, but  the degree morphism
is fixed and independent of $\KK$.

An ideal $I \subset S$ is homogeneous (with respect to the $A$-grading) if $I$ is generated by homogeneous elements,
that is polynomials, whose all monomials are of the same degree.
Whenever $A$ and the degree map are clear, we simply say ``homogeneous''.

Fix an $A$-grading of $S$. We consider a \emph{flat family of homogeneous ideals} in
  $S= \KK[\fromto{\alpha_1}{\alpha_m}]$.
That is, let $U$ be a $\KK$-scheme, and consider the free sheaf $S_U:=\ccO_U \otimes_{\KK} S$ of $A$-graded algebras. Let $\ccI \subset S_U$ be a homogeneous ideal sheaf flat over $U$. For each point $u\in U$ we consider the fibre ideal $\ccI_u\subset \kappa(u)[\fromto{\alpha_1}{\alpha_n}]= \kappa(u) \otimes_{\KK} S$.
Let $J\subset S$ be a fixed homogeneous ideal, and for each $u$ let $J_u:= \kappa(u) \otimes_{\KK} J \subset \kappa(u) \otimes_{\KK} S$ be the extension of $J$ to the coefficients in $\kappa(u)$. We say that an ideal $I\subset \kappa(u) \otimes_{\KK} S$ is saturated with respect to $J$ if $(I: J_u)= I$.

Under some minor assumptions the subset  of those $u\in U$ that $\ccI_u$ is saturated with respect to $J$ is Zariski open in $U$,
as explained in the following theorem.
\begin{thm}\label{thm_saturation_is_open_property}
   Suppose $S$ is a graded polynomial ring and
      $\ccI$ is a flat family of homogeneous ideals in $S$
      parametrised by a locally Noetherian $\KK$-scheme $U$.
   Let $J$ be a fixed homogeneous ideal in $S$.
   Then the set
        \[
            \left\{ u\in U \mid  \ccI_u \text{ is saturated with respect to $J$} \right\}
        \]
        is an open subset of $U$.
\end{thm}

For the proof see \cite[Thm~2.3]{jelisiejew_mandziuk_limits_of_saturated_ideals}.
The most relevant case of theorem is when $U$ is the multigraded Hilbert scheme.

\begin{cor}
  Suppose that $S$ is a positively graded polynomial ring and that $J \subset S$ is a homogeneous ideal.
  Let $\Hilb^h_S$ be the multigraded Hilbert scheme of $S$ for some Hilbert function $h\colon A\to \NN$.
  Consider the subset $\Hilb_S^{h, \sat J} \subset \Hilb^h_S$
    consisting of the points representing ideals saturated with respect to $J$.
  Then $\Hilb_S^{h, \sat J}$ is Zariski open.
\end{cor}

We now proceed to show that a general fibre of a flat family of schemes determines the whole family (Lemma~\ref{lem_flat_family_determined_by_generic_fibre}).

\begin{lemma}
\label{lem_saturated_is_preserved_by_localisation}
   Suppose $B$ is a comutative ring, $M\subset B$ is a multiplicative system, $I \subset B$ is an ideal,
   and $J \subset B$ is a finitely generated ideal.
   If $(I:J)=I$ then $(M^{-1} I : M^{-1} J) = M^{-1} I$ in the localisation $M^{-1}B$.
\end{lemma}
\begin{prf}
   We always have $(M^{-1} I : M^{-1} J) \supset M^{-1} I$.
   To show the opposite inclusion,
   suppose that $J= (\fromto{g_1}{g_k})$ and pick any $f\in (M^{-1} I : M^{-1} J)$.
   Thus for each $g_i$ there exists $h_i \in M$ such that $ f g_i h_i \in I$.
   Let $h = h_1\cdot h_2 \dotsm h_k$ so that also
   $ f g_i h \in I$ for each index $i$.
   Therefore $f h\in (I : J) = I$.
   Since $h\in M$ we have that $f\in M^{-1} I$.
\end{prf}

\begin{lemma}
   \label{lem_flat_family_determined_by_generic_fibre}
   Let  $\phi\colon \ccZ\to U$ be a flat morphism of schemes, where $U$ is integral and Noetherian.
   Let $\eta \in U$ be the generic point and $\ccZ_{\eta}$  the generic fibre of $\phi$.
   Then the scheme theoretic closure $\overline{\ccZ_{\eta}}$ is equal to $\ccZ$.
\end{lemma}
Note that the conclusion of the lemma can fail if $U$ is not integral or if $\phi$ is not flat.
\begin{prf}
   We always have $\overline{\ccZ_{\eta}} \subset \ccZ$.
   Also if $\ccZ=\emptyset$,  the claim is trivial. For the rest of the proof we suppose that $\ccZ$ is nonempty.

   To show the opposite inclusion, as in a typical Hartshorne's style exercise,
   we first reduce to an affine case, and then consider the problem algebraically.

\paragraph{Reduction to affine $U$.}
   Assume that the inclusion
   $\ccZ \subset \overline{\ccZ_{\eta}}$ holds whenever $U$ is in addition affine, and now let us deduce the case with arbitrary $U$.

   Suppose $U'\subset U$ is any nonempty affine open subset.
   In particular, $U'$ is dense and $\eta \in U'$.
   Moreover, the restricted map $\phi'\colon \ccZ|_{U'} \to U'$
   is flat by \cite[Prop.~III.9.2(b)]{hartshorne}.
   By our assumption, we have
     $\ccZ|_{U'}\subset \overline{\left(\ccZ|_{U'}\right)_{\eta}}$
   and we have the commutative diagram:
   \[
     \begin{tikzcd}
        \ccZ_{\eta} \arrow[r] \arrow[d]
          & \ccZ_{U'} \arrow[r, hook, "\text{open}"] \arrow[d]
          & \ccZ \arrow[d]\\
        \set{\eta} \arrow[r]
          & U'  \arrow[r, hook, "\text{open}"]
          & U.
      \end{tikzcd}
   \]
   Therefore
   $\overline{\left(\ccZ|_{U'}\right)_{\eta}} \subset \overline{\ccZ_{\eta}}$ and thus $\ccZ|_{U'}\subset \overline{\ccZ_{\eta}}$.
   Since $\ccZ|_{U'}$ cover $\ccZ$ when we vary $U'$ among open affine subsets of $U$, we obtain $\ccZ \subset \overline{\ccZ_{\eta}}$ as claimed.

\paragraph{Reduction to affine $\ccZ$.}
   Now assume that the inclusion
   $\ccZ \subset \overline{\ccZ_{\eta}}$ holds whenever $\ccZ$ is in addition affine, and let us deduce the case of arbitrary $\ccZ$.

   Let $\ccZ'\subset \ccZ$ be an affine open subset with the induced scheme structure.
   Then the composition $\phi' \colon \ccZ'\to U$ is also flat by \cite[Prop.~III.9.2(a) and (c)]{hartshorne}.
   The generic fibre of $\phi'$ is $(\ccZ')_{\eta}$,
   and by the affine case $\ccZ'\subset \overline{(\ccZ')_{\eta}}$.
   Moreover, $\overline{(\ccZ')_{\eta}}\subset \overline{\ccZ_{\eta}}$ by the commutativity of this diagram:
   \[
     \begin{tikzcd}
        (\ccZ')_{\eta}
         \arrow[rr] \arrow[ddr] \arrow[dr]
          & &\ccZ'
          \arrow[dr, hook, "\text{open}"] \arrow[ddr]
\\
      & \ccZ_{\eta}
      \arrow[rr]\arrow[d]
      &&\ccZ
     \arrow[d]
      \\
       & \set{\eta}
       \arrow[rr]
          && U.
      \end{tikzcd}
   \]
  Therefore,
  $\ccZ'\subset \overline{\ccZ_{\eta}}$
  and varying $\ccZ'$ we get an open covering of $\ccZ$ concluding that $\ccZ \subset \overline{\ccZ_{\eta}}$ as claimed.

\paragraph{Translating to algebraic condition.}
   Suppose $\ccZ$ and $U$ are affine.
   We claim that
   that $\overline{\ccZ_{\eta}} = \ccZ$ if and only if:
   \begin{equation}\label{equ_equality_for_generic_fibre_algebraically}
    \forall_{f\in \ccO_{\ccZ}(\ccZ),\   g\in \ccO_{U}(U)} \quad \bigl( \phi^*(g) \cdot f =0 \Rightarrow (f=0 \text{ or } g=0)\bigr)
   \end{equation}
   Indeed, we always have $\overline{\ccZ_{\eta}} \subset \ccZ$,
   while the other inclusion
   is equivalent to vanishing of the ideal of $\overline{\ccZ_{\eta}}$ in $\ccO_{\ccZ}(\ccZ)$.
   The ideal consists of those $f\in \ccO_{\ccZ}(\ccZ)$ that $\phi^*g\cdot f = 0$ for some non-zero $g\in \ccO_U(U)$.

\paragraph{Proving the algebraic conditon.}
Since the morphism $Z \to U$ is flat   and both $Z$ and $U$ are affine,
  $\ccO_{\ccZ}(\ccZ)$ is a flat $\ccO_{U}(U)$-module.
  Suppose that $g\in \ccO_{U}(U)$ is non-zero,   then the localisation $\ccO_{U}(U) \to \ccO_{U}(U)[g^{-1}]$
  is injective since $U$ is integral. Then, by definition of a flat module, also the map
\[
   \xi\colon \ccO_{U}(U)\otimes_{\ccO_{U}(U)} \ccO_{\ccZ}(\ccZ) \to \ccO_{U}(U)[g^{-1}]\otimes_{\ccO_{U}(U)} \ccO_{\ccZ}(\ccZ)
\]
is injective.
Now we prove that \eqref{equ_equality_for_generic_fibre_algebraically} holds.
Indeed, pick any $f\in \ccO_{\ccZ}(\ccZ)$ and $g\in \ccO_{U}(U)$ such that $\phi^*(g) \cdot f=0$.
If $g=0$, then we are done.
If $g\neq 0$, then consider
\[
  \xi(f) = g^{-1} \cdot \xi(\phi^*(g) \cdot f) = 0.
\]
Since $\xi$ is injective, we must have $f=0$, so the claim of \eqref{equ_equality_for_generic_fibre_algebraically} follows.
This also concludes the proof of the lemma.
\end{prf}

\begin{lemma}
\label{lem_adding_one_general_point_for_single_L}
Suppose $X$ is a variety over $\kk$ and $Z\subset X$ is a closed subscheme.
Let $L$ be a line bundle on $X$ and $p\in X$ be a general point.
Then one of the following conditions hold:
\begin{itemize}
 \item  $H^0(X, \ccI_Z\otimes L) = 0$ or
 \item  $H^0(X, \ccI_{Z\cup \set{p}} \otimes L)$ is of codimension $1$ in $H^0(X, \ccI_Z\otimes L)$.
\end{itemize}
\end{lemma}
\begin{prf}
   We have $H^0(X, \ccI_{Z\cup \set{p}} \otimes L) =
   H^0(X, \ccI_{Z} \otimes L) \cap H^0(X, \ccI_{\set{p}} \otimes L)$.
   Moreover, $H^0(X, \ccI_{\set{p}} \otimes L) \subset H^0(X, L)$ is of codimension at most $1$.
   Thus also
   $H^0(X, \ccI_{Z\cup \set{p}} \otimes L)$ is of codimension at most $1$ in $H^0(X, \ccI_Z\otimes L)$.
   Suppose that the latter codimension is $0$,
   that is
   $H^0(X, \ccI_{Z\cup \set{p}} \otimes L)=H^0(X, \ccI_Z\otimes L)$.
   Then any section of $H^0(X, \ccI_Z\otimes L)$ vanishes at a general point of $X$, thus vanishes everywhere on $X$, that is this section is $0$, so that $H^0(X, \ccI_Z\otimes L)=0$.
   Otherwise the codimension is $1$ proving the claim of the lemma.
\end{prf}

\subsection{Smooth projective toric varieties}
\label{sec_toric_varieties}

Suppose $X$ is a smooth toric variety over $\kk$
   and let $S=S[X] = \bigoplus_{D\in \Pic(X)} H^0(X,D)$
   be its Cox ring graded by $\Pic(X)$.
   Throughout this section we assume in addition that each $D\in \Pic(X)$ the $\kk$-linear space $H^0(X,D)$
   is finitely dimensional, which is true for instance if $X$ is projective (our main interest), or at least complete,
   but also in more general situations.
   We remark that for many  statements below smoothness of $X$ is not essential, but makes both
    statements and the proofs simpler.

    We study families of subschemes of $X$. Our main interest is in flat families over integral base,
    but some intermediate results are valid in more general setting.
    Moreover, we  frequently restrict to the case of affine base.
Thus let $U$ be a $\kk$-scheme, and suppose $\ccZ\subset U \times_{\kk} X$ is a closed subscheme.
In this situation, we say that $\ccZ$ is a (not necessarily flat) family of subschemes of $X$ parametrised by $U$.
Denote by $\pi_U\colon U\times X \to U$ and
$\pi_X\colon U\times X \to X$ the natural projections.
There are two related ideal sheaves we consider:
\begin{itemize}
 \item  the ideal sheaf of ${\ccZ}$ embedded in the product: $\ccI_{\ccZ}\subset \ccO_{U \times X}$, or in the affine case $U=\Spec A$, $\ccI_{\ccZ} \subset A\otimes_{\kk} \ccO_{X} $ and
 \item  the homogeneous ideal of $\ccZ$: $I_{\ccZ} \subset \ccO_U \otimes_{\kk} S_X$, or in the affine case, $I_{\ccZ} \subset A\otimes_{\kk} S_X$.
\end{itemize}

\begin{rem}
    \label{rem_relations_ideal_sheal_and_homogeneous_ideal}
    \begin{enumerate}
     \item
      To distinguish the two objects above,  we skip the word ``sheaf'' in the description of $I_{\ccZ}$,
        despite it is really a sheaf of ideals,  and only in the affine case it reduces to an honest ideal in a ring.
     \item
        \label{item_relation_between_ideal_sheal_and_homogeneous_ideal}
        The relation between $\ccI_{\ccZ}$ and $I_{\ccZ}$    is given by the following formula
    \[
       \forall_{D\in\Pic X}  \ (I_{\ccZ})_D
       = \pi_{U*}(\ccI_{\ccZ}\otimes_{\ccO_{U \times X}} \pi_X^*D).
    \]
    \item
      Suppose $\phi \colon  U'\to U$ is a morphism of $\kk$-scheme and define by $\ccZ'\subset U'\times_{\kk} X$
      the pullback family.
      While the ideal sheaf $\ccI_{\ccZ}$ behaves functorially, that is $(\phi\times \id_X)^* \ccI_{\ccZ} = \ccI_{\ccZ'}$,
      the homogeneous ideal is not functorial.
       It is always true that $\phi^*I_{\ccZ} \subset I_{\ccZ'}$,
       but $I_{\ccZ'}$ is the saturation of $\phi^*I_{\ccZ}$ with respect to the irrelevant ideal $\Irrel{X}$ discussed below.
    \end{enumerate}
\end{rem}

An important special case is when $U=\Spec \KK$ for a field extension $\kk\subset \KK$,  and
   $Z\subset X\times_{\kk} \KK$ is a closed subscheme.
Then $I_Z \subset S[X]\otimes_{\kk}\KK$ is the homogeneous ideal of $Z$, that is the ideal generated by sections of line bundles $D$ on $X\times_{\kk} \KK$ that vanish on $Z$.
The \emph{Hilbert function} of $Z$ is $h_Z\colon \Pic(X) \to \NN$ defined by
\[
h_Z(D) = \dim_{\KK} \left(\left(S[X]\otimes_{\kk}\KK\right)/I_{Z}\right)_D.
\]
Thus if $\ccU$ is any $\kk$-scheme
and  $\ccZ \subset U \times_{\kk} X$ is a family of subschemes of $X$, then for each point $u \in U$ we obtain a function $h_{\ccZ_u} \colon \Pic(X)\to \NN$, namely the Hilbert function of the fibre $\ccZ_u$ (here $\KK = \kappa (u)$).

Define $\Irrel{X}\subset S$ to be the irrelevant ideal,
   that is the radical ideal of $S_L$ for any very ample line bundle $L$ on $X$.
Throughout the rest of the article we fix $\Irrel{X}$ as the denominator of the saturation,
  that is, we whenever we say an ideal (or ideal sheaf) is saturated, we mean saturated with respect to $\Irrel{X} \subset S[X]$
  or $\Irrel{X} \cdot \KK \subset \KK \otimes_{\kk} S[X]$
  or
  $\Irrel{X} \cdot \ccO_U \subset \ccO_U \otimes_{\kk} S[X]$.
Note that for any subscheme $Z\subset \Spec \KK \times_{\kk}X$ its ideal $I_Z$ is saturated, and similarly, $I_{\ccZ}$ is saturated for any family
$\ccZ$ of subschemes of $X$ parametrised by $U$.

Conversely, for any homogeneous ideal  $I\subset \ccO_U \otimes_{\kk} S_X$ we can define the underlying closed subscheme $\ccZ(I)$ of $U\times X$.
 Formally, we define $\ccZ(I)$ by constructing its ideal sheaf $\ccI\subset \ccO_{U\times X}$.
For any open affine subset $X^{\circ}\subset X$, the complement $X\setminus
X^{\circ}$ is a zero locus of some $\Theta\in S[X]_L$ for some $L\in \Pic(X)$, and $X^{\circ}= \Spec(S[X][\Theta^{-1}]_0)$.
If $U^{\circ}\subset U$ is any open subset
then we set
\[
  \ccI(U^{\circ}\times X^{\circ}):=I(U^{\circ})[\Theta^{-1}]_0\subset \ccO_{U}(U^{\circ})\otimes_{\kk} S[X][\Theta^{-1}]_0 = \ccO_{U\times X}(U^{\circ}\times X^{\circ}).
\]
It is a formal check that there exists a unique coherent sheaf of ideals $\ccI\subset \ccO_{U\times X}$ which takes the above values on any product
open subsets $U^{\circ}\times X^{\circ}$ (with $X^{\circ}$ affine).
We define $\ccZ(I)$ to be the closed subscheme defined by such $\ccI$.
Moreover, the following properties are also straightforward to verify.
\begin{prop}
\label{prop_zero_scheme_of_ideal}
   Suppose $I$ and $\ccZ(I)$ are as above. Denote by $I^{\sat}$ the saturation of $I$ with respect to $\Irrel{X}$. Then:
   \begin{enumerate}
    \item \label{item_zero_scheme_of_ideal_and_its_saturation}
         $\ccZ(I)=\ccZ(I^{\sat})$, and
    \item
    \label{item_ideal_of_zero_scheme}
       $I_{\ccZ(I)} = I^{\sat}$.
    \item
    \label{item_two_ideals_have_the_same_zero_scheme}
       if $L$ is an ample line bundle and $J$ is another homogeneous ideal in $\ccO_U \otimes_{\kk} S_X$, then
    $\ccZ(I) = \ccZ(J)$ if and only if there exists $t_0\geqslant 0$ such that for all $t\geqslant t_0$ we have $I_{L^{t}} = J_{L^{t}}$.
   \end{enumerate}
\end{prop}
\noprf

\subsection{A family has finitely many  Hilbert functions}

The main goal of this subsection is to prove that a general configuration of points has maximal possible Hilbert function (Theorem~\ref{thm_Sip_open}).
Along the way we prove several more general statements that might be of an independent interest, particularly Theorem~\ref{thm_structure_of_Hilbert_functions_of_fibres} about stratification of a family of subschemes by locally closed loci with fixed Hilbert functions

\begin{lemma}
   \label{lem_ideal_of_restriction_is_restriction_of_ideal_on_open}
   Suppose $U$ is an integral Noetherian  $\Bbbk$-scheme of finite dimension. Let  $\ccZ \subset U \times_{\Bbbk} X$ be a family of subschemes as in \S\ref{sec_toric_varieties}.
   Then there exists an open dense subset
   $V\subset U$ such that for all $v\in V$
   we have
   $I_{\ccZ_{v}} = I_{\ccZ} \otimes_{\ccO_U}\kappa(v)$.
\end{lemma}

\begin{prf}
   By generic flatness  \cite[\href{https://stacks.math.columbia.edu/tag/052A}{Tag 052A}]{stacks_project} we may restrict to an open dense subset of $U$, and suppose that
   $I_{\ccZ}$ is flat over $U$.
   Let $\eta\in U$ be the generic point.
   Since the ideal $I_{\ccZ}$ is saturated, by Lemma~\ref{lem_saturated_is_preserved_by_localisation}
   also the ideal $I_{\ccZ} \otimes_{\ccO_U}\kappa(\eta) \subset S[X] \otimes_{\kk} \kappa(\eta)$
   is saturated.
   Therefore by Theorem~\ref{thm_saturation_is_open_property}
   also $I_{\ccZ} \otimes_{\ccO_U}\kappa(v) \subset S[X] \otimes_{\kk} \kappa(v)$
   is saturated for all $v$ in an open neighbourhood $V$ of the generic point $\eta$.
   But the saturation of $I_{\ccZ} \otimes_{\ccO_U}\kappa(v)$ is precisely equal to $I_{\ccZ_{v}}$ by Proposition~\ref{prop_zero_scheme_of_ideal}\ref{item_ideal_of_zero_scheme} proving the claim.
\end{prf}

\begin{thm}
   \label{thm_structure_of_Hilbert_functions_of_fibres}
   Suppose $U$ is a Noetherian finitely dimensional
   $\kk$-scheme and $\ccZ \subset U \times_{\kk} X$ is a family of subschemes of $X$ parametrised by $U$. Then
   \begin{enumerate}
    \item  \label{item_finitely_many_Hilbert_functions_in_a_family}
       there are finitely many functions   $h_i \colon \Pic X \to \NN$
       where  $i = \fromto{1}{k}$,  such that for each point $u \in U$
       the Hilbert function $h_{\ccZ_u}= h_i$ for some $i\in \fromto{1}{k}$,
    \item
    \label{item_Hilbert_function_stratification_is_locally_closed}
       the subset $U_i\subset U$ with $U_i=\set{u\in U\mid h_{\ccZ_u} = h_i}$ is locally closed,
       and $h_i \leqslant h_j$ whenever $U_i \subset \overline{U_j}$,
    \item \label{item_generic_Hilbert_function_is_open}
      if $U$ is in addition integral, then there exists an open dense subset $U^{\circ} \subset U$ such that $u\in U^{\circ}$
      if and only if $h_{\ccZ_u} = h_{\ccZ_{\eta}}$, where $\eta \in U$ is the generic point.
   \end{enumerate}
\end{thm}

\begin{prf}
   Without loss of generality we may assume $U$ is reduced,
   as the statements are all set-theoretic and independent of the scheme structure of $U$.

   We first prove a weaker version of \ref{item_generic_Hilbert_function_is_open}:
   \renewcommand{\theenumi}{(\roman{enumi}')}
   \begin{enumerate}
    \addtocounter{enumi}{2}
    \item\label{item_generic_Hilbert_function_is_dense}
     If $U$ is in addition integral, then there exists an open dense subset $U^{\circ} \subset U$ such that
     $h_{\ccZ_u} = h_{\ccZ_{\eta}}$ for each $u\in U^{\circ}$.
   \end{enumerate}
   \renewcommand{\theenumi}{(\roman{enumi})}

    To prove \ref{item_generic_Hilbert_function_is_dense}, by generic flatness
    \cite[\href{https://stacks.math.columbia.edu/tag/052A}{Tag 052A}]{stacks_project} we may assume that the map
    $\ccZ\to U$ is flat by restricting to a suitable open dense subset of $U$ if necessary.
    By Lemma~\ref{lem_ideal_of_restriction_is_restriction_of_ideal_on_open}
    we may also assume that
     $I_{\ccZ_{u}} = I_{\ccZ} \otimes_{\ccO_U}\kappa(u)$ for all $u \in U$.
    After these two reductions, we claim we can  take $U^{\circ}=U$.
    Then for each $D\in \Pic(X)$ the $D$-th grading $(I_{\ccZ})_D$ is a flat finitely generated $\ccO_U$-module, hence it is locally free, and therefore (since $U$ is integral) the dimensions of $(I_{\ccZ})_D \otimes_{\ccO_U}\kappa(u)$ are independent of $u\in U$.
    Thus also the dimensions of $(I_{\ccZ_{u}})_D$
    are independent of $u\in U$ as claimed in \ref{item_generic_Hilbert_function_is_dense}.

    To prove \ref{item_finitely_many_Hilbert_functions_in_a_family} we argue by induction on the dimension and number of irreducible components of $U$.
    Note that both these are finite by our assumptions.
    If $\dim U=0$, then the claim is clear.
    So suppose we know the statement of \ref{item_finitely_many_Hilbert_functions_in_a_family} for all bases of dimension at most $k-1$,
    and now suppose the dimension of $U$ is equal to $k$. If $U$ is not irreducible, then we argue for each component separately.
    Thus it is enough to assume $U$ is irreducible (and therefore integral).
    By \ref{item_generic_Hilbert_function_is_dense}
    there is an open dense subset $U^{\circ} \subset U$ with only one Hilbert function.
    By the induction assumption there are only finitely many Hilbert functions on
    $U\setminus U^{\circ}$, which concludes the proof of \ref{item_finitely_many_Hilbert_functions_in_a_family}.

   To prove \ref{item_Hilbert_function_stratification_is_locally_closed} we consider the partial order on the set of Hilbert functions $\setfromto{h_1}{h_k}$:
    $h_i\leqslant h_j$ if and only if all the values are less or equal.
    Then any minimal element of the set corresponds to a closed stratum and we prove the claim by induction on $k$ removing this stratum.

    Finally, \ref{item_generic_Hilbert_function_is_open} follows from \ref{item_Hilbert_function_stratification_is_locally_closed}, as one of the strata $U_i$ must contain $\eta$. Its closure is therefore the whole $U$, and thus $U_i$ is open in $U$.
    Set $U^{\circ}:=U_i$, which satisfies the required properties.
\end{prf}

\begin{cor}
\label{cor_adding_one_general_point_for_all_L}
Suppose $X$ is as above,
$Z\subset X$ is a closed subscheme, and $p\in X$ be a general point.
Then for all $D\in \Pic(X)$:
\[
  h_{Z\cup \set{p}}(D) = \min \set{ h_{Z}(D) +1, \dim_{\kk}S[X]_{D}}.
\]
\end{cor}
\begin{prf}
   If $Z=X$, then claim is clear.
   So suppose $Z$ is a proper subscheme and consider $U= X\setminus Z$ to be the set-theoretic complement of $Z$.
   Then in particular $U$ is integral Noetherian scheme of finite dimension.
   Consider the subscheme $\ccZ \subset U \times X$, which is the union of $U \times Z$ and the diagonal $\Delta \subset U\times X$.
   Then the fibre $\ccZ$ over $p\in U$ is precisely $Z \cup \set{p}$.
   For each single $D\in \Pic(X)$ and general $p\in U$
   (with the generality condition dependent of $D$)
   we have $h_{\ccZ_{p}}(D)= \min \set{ h_{Z}(D) +1, \dim_{\kk}S[X]_{D}}$
   by Lemma~\ref{lem_adding_one_general_point_for_single_L}.
   Thus $h_{\ccZ_{\eta}}(D)= \min \set{ h_{Z}(D) +1, \dim_{\kk}S[X]_{D}}$ for the generic point $\eta\in U$.
   By Theorem~\ref{thm_structure_of_Hilbert_functions_of_fibres}\ref{item_generic_Hilbert_function_is_open}
   there is an open subset $U^{\circ}\subset U$ such that the Hilbert function over $U^{\circ}$
   is fixed, end equal to the function at the general point $\eta$, proving the claim of the corollary.
\end{prf}

Fix an integer $r>0$.
As in \cite[\S3.2]{nisiabu_jabu_border_apolarity}
  by $h_{r, X}\colon \Pic X \to \NN$
   denote the
  \emph{generic Hilbert function of $r$ points on $X$},
  that is $h_{r, X}(D) = \min \left( r, \dim_{\kk} H^0(X,D)\right)$.
In \cite[Lem.~3.9]{nisiabu_jabu_border_apolarity} we show that (over complex numbers) a very general configuration of $r$ points in X has the Hilbert function $h_{r, X}$.
Here we show that the same is actually true for a general configuration. Moreover, we extend the result to an arbitrary algebraically closed base field.

\begin{thm}\label{thm_Sip_open}
  Suppose
    $X$ is a smooth projective toric variety over $\kk$.
    Suppose that $Z:=\setfromto{x_1}{x_r}$ is a general configuration of points on $X$ (that is, $(\fromto{x_1}{x_r}) \in X^{\times r}$ is a general point).
    Then the Hilbert function $h_Z$ is equal to $h_{r, X}$.
  Moreover, the set of such $r$-tuples of points
     is open in $X^{\times r}$.
\end{thm}

\begin{prf}
   The claim about the Hilbert function follows by induction from Corollary~\ref{cor_adding_one_general_point_for_all_L}.
   The openness claim follows from Theorem~\ref{thm_structure_of_Hilbert_functions_of_fibres}\ref{item_generic_Hilbert_function_is_open} applied to $U=X^{\times r}$ and $\ccZ\subset U \times X$ equal to
   \[
      \ccZ = \set{(\fromto{x_1}{x_r} , x) \mid x=x_i\text{ for some } 1\leqslant i \leqslant k}.
   \]
\end{prf}

\section{Hilbert schemes, Grassmann relative linear span and cactus varieties}
\label{sec_cactus_variety}

In this section we compare two types of Hilbert schemes.
The ``usual'' one $\usualHilb(X)$, introduced by Grothendieck \cite{grothendieck_techn_de_constr_et_exist_GA_IV_Hilb}, parametrises flat families of closed subschemes of a given projective variety $X$.
The ``multigraded'' one $\Hilb(S)$, introduced by Haiman and Sturmfels \cite{haiman_sturmfels_multigraded_Hilb}, parametrises families of multihomogeneous ideals in a graded ring $S$ with a locally constant Hilbert function.

The main setting in which both concepts appear is (similarly to Subsection~\ref{sec_toric_varieties}) when $X$ is a smooth projective toric variety and $S=S[X]$ is its Cox ring (again smoothness of $X$ is not critical, as in Subsection~\ref{sec_toric_varieties}).
Then both schemes $\Hilb(S)$ and $\usualHilb(X)$ make sense and have projective connected components.
We will see that they are related by a natural morphism $\sch\colon \Hilb(S) \to \usualHilb(X)$ which is surjective on points.
Moreover, for each irreducible component of $\usualHilb(X)$ there exists a unique irreducible component of $\Hilb(S)$ such that the underlying morphism of varieties is birational.

\subsection{Grothendieck and multigraded Hilbert schemes}
\label{sec_Hilb_schemes}

Comprehensive and detailed definitions, constructions, and descriptions of properties of $\usualHilb(X)$ can be found in numerous excellent textbooks including
\cite[Chapt.~5]{fantechi_et_al_fundamental_ag},
\cite[Chapt.~4]{jelisiejew_PhD},
\cite[Chapt.~4]{sernesi_Deformations_of_algebraic_schemes}.
See also \cite[Chapt.~I]{kollar_book_rational_curves}.
In our setting, let $L$ be any very ample line bundle on $X$, then $\usualHilb(X)$ is a disjoint union $\bigsqcup_{p\in \QQ[t]} \usualHilb_{L,p}(X)$ where $p$ is a polynomial in a single variable with integral values for integer arguments
and $\usualHilb_{L,p}(X)$ is a scheme parameterising subschemes of $X$
with Hilbert polynomial $p$ (that is, the Euler characteristics of $L^{\otimes t}$ restricted to the subscheme is equal $p(t)$).
In this setting (with $X$ not necessarily isomorphic to a projective space), it is not necessary that $\usualHilb_{L,p}(X)$ is connected.
Moreover, different choices of $L$ might lead to different decompositions into pieces according to Hilbert polynomials, in any case, for each $L$ and each $p$, $\usualHilb_{L,p}(X)$ is always a projective scheme over $\kk$ \cite[Thm~I.1.4]{kollar_book_rational_curves}.

On the other hand the multigraded Hilbert scheme $\Hilb S[X]$
is less standard and not so often considered or exploited, perhaps due to the larger complexity and pathologies appearing on much smaller cases.
We also have a disjoint union decomposition
$\Hilb S[X] = \bigsqcup_{h \colon \Pic X \to \NN} \Hilb_h S[X]$, where $h$ is any function and $\Hilb_h S[X]$ parametrises all families of multihomogeneous ideals with Hilbert function $h$. Again, each $\Hilb_h S[X]$ is projective though not necessarily connected.

\begin{notation}[elements of parameter spaces]
\label{not_parameter_spaces}
  For a space parametrising some type of objects
    we need to distinguish an element of the space $[\bullet]$ and the object $\bullet$ it represents.
  Explicitly, we will use the Grothendieck Hilbert scheme $\Hilb(S)$, multigraded Hilbert schemes $\usualHilb(X)$, and Grassmannians $\Gr(i, V)$, or their subsets.
  \begin{itemize}
   \item For a homogeneous ideal $I \subset S$
           we  denote by $[I]$ the corresponding $\kk$-point of $\Hilb(S)$, and vice versa: if $[I] \in \Hilb(S)(\kk)$, then implicitly $I\subset S$ is the ideal represented by $[I]$.
   \item For a subscheme $R \subset X$
           we  denote by $[R]$ the corresponding $\kk$-point of $\usualHilb(X)$, and vice versa as above.
   \item If $V$ is a vector space, then for a linear subspace $E \subset V$ with $\dim E =i$
           we denote by $[E]\in \Gr(i, V)(\kk)$ the corresponding $\kk$-point, and vice versa.
  \end{itemize}
  To avoid confusion, we refrain from using this convention for non-closed points of the parameter space.
\end{notation}

\begin{prop}
   There exists a morphism of schemes $\sch\colon \Hilb S[X] \to \usualHilb X$ which on points takes a homogeneous ideal $I\subset S[X]$ to the scheme $Z(I) \subset X$, $\sch([I]) = [Z(I)]$.
\end{prop}

A similar statement with a slightly different notation can be found in \cite[\S2.4.1]{jelisiejew_mandziuk_limits_of_saturated_ideals}.
We include the proof for the sake of completeness.
\begin{prf}
   We argue componentwise for each Hilbert function $h \colon \Pic X \to \NN$ separately.
   If $\Hilb_h S[X]  =\emptyset$, then there is nothing to do.
   Thus suppose $\Hilb_h S[X] \neq \emptyset$ and fix a very ample line bundle $L$.
   In particular, pick any $[I]\in \Hilb_h S[X]$
   the Hilbert polynomial $p(t)$ of $Z(I)$ with respect to $L$ agrees with $h(L^{t})$ for $t\gg 0$.
   Therefore $p$ is determined by $h$ (and $L$), and does not depend on the choice of $I$.
   To prove the proposition it is enough to construct $\sch\colon \Hilb_h S[X] \to \usualHilb_{L,p} X$.

   The universal ideal
   $J\subset \ccO_{\Hilb_h S[X]} \otimes_{\kk} S[X]$ is a homogeneous ideal that
   defines a closed subscheme $\ccR:=\ccZ(J)\subset \Hilb_h S[X]\times X$.
   Consider also $I_\ccR\subset \ccO_{\Hilb_h S[X]} \otimes_{\kk} S[X]$, the homogeneous ideal of $\ccR$.
   By Proposition~\ref{prop_zero_scheme_of_ideal}\ref{item_ideal_of_zero_scheme}
   we have $I_\ccR = J^{\sat}$ and by
   Proposition~\ref{prop_zero_scheme_of_ideal}\ref{item_two_ideals_have_the_same_zero_scheme}
   for $t\gg 0$ we have:
   \[
     J_{L^t} = (I_\ccR)_{L^t}  \stackrel{\text{by Remark~\ref{rem_relations_ideal_sheal_and_homogeneous_ideal}
         \ref{item_relation_between_ideal_sheal_and_homogeneous_ideal}} }{=} (\pi_{\Hilb_h S[X]})_* \left(\ccI_{\ccR}\otimes \pi_X^*(L^{t})\right)
   \]
   In particular, since $J$ is a flat $\ccO_{\Hilb_h S[X]}$-module, also its direct summands are flat, and hence $(\pi_{\Hilb_h S[X]})_* \left(\ccI_{\ccR}\otimes \pi_X^*(L^{t})\right)$ is flat and finite module, hence locally free of rank $p(t)$.
   Thus $\ccR$ is
     flat over $\Hilb_h S[X]$
   by \cite[Prop.~III.9.9]{hartshorne} or \cite[Lem.~5.5]{fantechi_et_al_fundamental_ag} or \cite[Exercise~24.7.C(b)]{vakil_FoAG}.
   Therefore we have a map  $\sch\colon \Hilb_h S[X] \to \usualHilb_{L,p} X$ as claimed.
\end{prf}

Note that the map $\sch$ is surjective on $\kk$-points: if $[R]\in \usualHilb(X)$ is a $\kk$-point which represents a closed subscheme $R\subset X$, then the multihomogeneous ideal $I(R) \subset S[X]$ is an element of $\Hilb(S[X])$ such that $\sch([I(R)]) = [R]$.
However, typically $\sch$ is \emph{not}
a projective morphism,
but only for a silly reason: there could be infinitely many connected components of $\Hilb(S[X])$ that are mapped onto the same connected component of $\usualHilb(X)$.
Once restricted to suitable Hilbert function $h \colon \Pic X\to \NN$ and suitable Hilbert polynomial $p$, the map
$\sch\colon \Hilb_h S[X] \to \usualHilb_{L,p} X$ is a morphism of projective schemes, hence projective.

In the following couple of statements we disscuss  a partial inverse of the morphism $\sch$, at least for the reduced structures.

\begin{notation}\label{notation-generic_Hilbert_function}
Suppose $\ccH\subset \reduced{\usualHilb(X)}$ is any irreducible component.
Let $\ccZ \subset X \times \ccH$ be the natural family of subschemes of $X$ parametrised by $\ccH$.
Suppose also that $\ccH^{\circ}\subset \ccH$ is the open dense subset of $\ccH$
  --- as in Theorem~\ref{thm_structure_of_Hilbert_functions_of_fibres}\ref{item_generic_Hilbert_function_is_open} ---
  where the Hilbert function of $R=\ccZ_{[R]}$ for $[R]\in \ccH^0$ is equal to the function  $h_{\ccH}$.  
\end{notation}

\begin{prop}
\label{prop_idealmap}
Let $\ccH$, $\ccH^0$ and $h_{\ccH}$ be as in the Notation~\ref{notation-generic_Hilbert_function}  as above,
 there is a morphism
  \[
    \idealmap \colon \ccH^{\circ} \to \Hilb_{h_{\ccH}}(S[X])
  \]
 such that $\sch \circ \idealmap = \id_{\ccH^{\circ}}$ and $\idealmap([R]) = [I(R)]$ for all $[R]\in \ccH^{\circ}$.
\end{prop}

\begin{prf}
  Consider the universal homogeneous ideal $I \subset \ccO_{\ccH^{\circ}} \otimes_{\kk} S[X]$, so that $I|_{\set{[R]}} = I(R)$.
  By the definition of $\ccH^{\circ}$ the Hilbert function of $I|_{\set{[R]}}$ is constant and equal to $h_{\ccH}$.
  Since $\ccH$ (and also $\ccH^{\circ}$) is reduced, for every $D\in \Pic(X)$ the sheaf $I_D$ is locally free of rank $h_{\ccH}(D)$ by \cite[Lem.~II.8.9]{hartshorne}.
  Thus by the universal property of the multigraded Hilbert scheme there exists a morphism $\idealmap \colon \ccH^{\circ} \to \Hilb_{h_\ccH}$
  such that $\idealmap([R]) = [I(R)]$ for every $[R]\in \ccH^{\circ}$.

  Moreover,
  \[
   \sch \circ \idealmap([R]) = \sch([I(R)]) = [Z(I(R))]= [R],
  \]
  and since $\ccH^{\circ}$ is reduced, this verification on points is sufficient to show that
  $\sch \circ \idealmap = \id_{\ccH^{\circ}}$, concluding the proof of the proposition.
\end{prf}

\begin{cor}
\label{cor_saturation_map_is_birational}
  The closure of the image
    $\overline{\idealmap(\ccH^{\circ})}$ is an irreducible component $\mathbf{H}_{\ccH}$ of $\reduced{\Hilb_{h_{\ccH}}(S[X])}$,
    and $\sch\colon\mathbf{H}_{\ccH} \to \ccH$ is a projective birational morphism.
\end{cor}

\begin{prf}
    The scheme $\mathbf{H}_{\ccH} \subset \Hilb_{h_{\ccH}}(S[X])$ is closed by the definition and irreducible since $\ccH^{\circ}$ is irreducible.
    Naturally, $\sch|_{\mathbf{H}_{\ccH}} \colon\mathbf{H}_{\ccH} \to \ccH$ is a projective morphism, and it admits birational inverse $\idealmap$ by Proposition~\ref{prop_idealmap}.
    Thus the only claim that remains to be proved is that $\mathbf{H}_{\ccH} \subset \Hilb_{h_{\ccH}}(S[X])$ is a component.

    Thus let $\mathbf{H} \subset \Hilb_{h_{\ccH}}(S[X])$ be any irreducible component containing $\mathbf{H}_{\ccH}$.
    We claim that $\mathbf{H} = \mathbf{H}_{\ccH}$.
    Indeed, since $\mathbf{H}_{\ccH}$ contains some saturated ideals, $\mathbf{H}_{\ccH}\subset \mathbf{H}$, and the set of saturated ideals is open (Theorem~\ref{thm_saturation_is_open_property}), there exists an open dense $U\subset \mathbf{H}$ with only saturated ideals.
    Now $\sch|_U$ is a locally closed immersion by \cite[Cor.~2.8]{jelisiejew_mandziuk_limits_of_saturated_ideals}.
    In particular, comparing the dimensions:
    \begin{alignat*}{2}
      \dim \mathbf{H}_{\ccH} & \leqslant \dim \mathbf{H} & \quad &(\text{since } \mathbf{H}_{\ccH}\subset \mathbf{H})\\
      &= \dim U && (\text{since $U$ is open and dense})\\
      &= \dim \sch(U) && (\text{since $\sch|_U$ is an immersion})\\
      &\leqslant \dim \ccH &&(\text{since } \sch(U)\subset \ccH)\\
      &\leqslant \dim \mathbf{H}_{\ccH} &&
      (\text{since $\sch\colon \mathbf{H}_{\ccH}\to \ccH$ is dominant}).
    \end{alignat*}
    So all the above dimensions are equal,
      in particular $\dim \mathbf{H}_{\ccH} = \dim \mathbf{H}$,
      and thus $\mathbf{H}_{\ccH} = \mathbf{H}$
    as claimed.
\end{prf}

\begin{proof}[Proof of Proposition~\ref{prop_distinguished_components}]
   The distinguished irreducible
   component $\mathbf{H}_{\ccH}$
   is defined in Corollary~\ref{cor_saturation_map_is_birational}.
   A general ideal $I\in \mathbf{H}_{\ccH}$
     is the image of $\idealmap([R])$ for some
     $[R]\in \ccH^{\circ}$, hence $I=I_R$ and thus $I$ is saturated by Proposition~\ref{prop_zero_scheme_of_ideal}\ref{item_ideal_of_zero_scheme}, as claimed in the first item the proposition.
   The second item holds by Corollary~\ref{cor_saturation_map_is_birational}.

   The uniqueness follows from Theorem~\ref{thm_saturation_is_open_property}:    suppose $\mathbf{H}'$ is another component of the multigraded
      Hilbert scheme satisfying the two itemised properties.
   Let $U' \subset \mathbf{H}'$ be the open subset of saturated ideals.
   Note that $U'$ is not empty by the first item.
   Then, the birationality from the second item implies
      that the general element of $\ccH$ is in $\sch(\mathbf{H}')$.
   Thus the map $\idealmap$ sends such general element to
      $\mathbf{H}'$ and thus $\mathbf{H}'=\mathbf{H}_{\ccH}$ as claimed.
\end{proof}

\subsection{Linear span map}
\label{sec_linear_span_map}
For a fixed finite dimensional $\kk$-vector space $V$,
let $\Gr(k, V)$ be the Grassmannian of $k$-dimensional linear subspaces of $V$.
Equivalently, $\Gr(k, V)$ parametrises projective linear subspaces of $\PP(V)$.
Thus if $U$ is any base scheme and $\ccE\subset U\times V$ is a vector subbundle of rank $k$, then there exists a regular map $U\to \Gr(k,V)$.
Here the vector bundle plays the role of a family of vector spaces parametrised by $U$.

Each subscheme of the projective space $\PP(V)$ has an associated projective linear subspace of $\PP(V)$ (thus a point in the Grassmannian $\Gr(k, V)$ for some $k$), namely the linear span of $R$, or more precisely, the smallest linear subspace $\linspan{R}\subset \PP V$ such that $R\subset \linspan{R}$ as subscheme.

Unfortunately, in general the assignment $R\mapsto \linspan{R}$ cannot be made into a morphism from a Hilbert scheme to the Grassmannians, as flat degeneration $R_t \stackrel{t\to 0}{\longrightarrow} R_0$ can have lower dimension of $\linspan{R_0}$ than that of  $\linspan{R_t}$ for general $t$ (hence this assignment is not even continuous).
This is an infamous obstruction to efficiently study the secant varieties of subvarieties of projective space.
In the literature there exist essentially two workarounds this obstruction.
The first one is to avoid settings when this obstruction actually  occurs: \cite{nisiabu_jabu_cactus}, \cite{nisiabu_jabu_farnik_cactus_Fujita}, \cite{choi_lacini_park_sheridan_sings_and_syz_of_secant_vars}, or at least prove that it does not affect your case \cite[Prop.~3.1]{ranestad_voisin_VSP_and_divisors_in_the_moduli_of_cubic_fourfolds}, \cite[Thm~9.2, Rem.~9.3]{jelisiejew_landsberg_pal_concise_tensors_of_minimal_border_rank}.
The second one is coming from usage of multigraded Hilbert scheme instead.
Over $\kk=\CC$ this latter approach is accomplished in  \cite{nisiabu_jabu_border_apolarity}.
Here we show that it also generalises to arbitrary $\kk$.

Fix a smooth projective toric variety $X$ and a  divisor $D\in \Pic(X)$.
For any Hilbert function $h \colon \Pic(X) \to \NN$ we can construct a vector subbundle $\ccE_{D} \subset \Hilb_h(S[X]) \times S[X]_D^*$ of rank $h(D)$
by taking the flat (hence locally free) subsheaf $I_{D}\subset \ccO_{\Hilb_h(S[X])} \otimes_{\kk} S[X]_D$ and defining $\ccE_{D}:= (I_{D})^{\perp}$.
Thus by the universal property of the Grassmannian, we obtain a morphism of schemes
$\perpmap_D \colon  \Hilb_h(S[X]) \to  \Gr(h(D), S[X]_D^*)$, and explicitly we have $\perpmap_{D}([J]) = [(J_D)^{\perp}$] for any $[J]\in \Hilb_h(S[X])$.
Below we assume that $D$ is effective, so that the projective space $\PP(S[X]_D^*)$ is non-empty.

\begin{prop}
   \label{prop_linear_span_map}
   Fix a line bundle $D\in \Pic(X)$
      and an irreducible component
      $\ccH\subset \reduced{\usualHilb(X)}$.
   Let $\iota\colon X \dashrightarrow \PP(S[X]_D^*)$
      be the rational map defined by the complete linear system of $D$,
      and let $\reg(\iota) \subset X$ be the open subset where $\iota$ is regular.
   Suppose that at least one of the following conditions holds:
   \begin{enumerate}
    \item
      \label{item_lin_span_very_ample}
    $D$ is very ample, or
    \item
      \label{item_lin_span_zero_dim}
     $\ccH$ parametrises zero dimensional subschemes, or
    \item
      \label{item_lin_span_associated_pts}
    for a general $[R] \in \ccH$, all associated points of $R$ are in $\reg(\iota)$.
   \end{enumerate}
   Then there exists an open dense subset $U\subset \ccH$
   and a morphism
   \[
      \spanmap_D\colon U\to Gr(h_{\ccH}(D), S[X]_D^*)
    \]
     such that for any $[R]\in U(\kk)$ we have
     $\PPof{\spanmap_D([R])}= \left[\linspan{\iota(R\cap \reg(\iota))}\right]$.
     Moreover,
        the rational map
      $\spanmap_D\colon \ccH \dashrightarrow \Gr(h_{\ccH}(D), S[X]_D^*)$
      is resolved by the precomposition with the map
      $\sch$ restricted to $\mathbf{H}_{\ccH}$, which is equal to $(\perpmap_D)|_{\mathbf{H}_{\ccH}}$.
\end{prop}

\begin{prf}
  First observe that both \ref{item_lin_span_very_ample} and \ref{item_lin_span_zero_dim} imply
  that~\ref{item_lin_span_associated_pts} holds.
   Indeed,   if $D$ is very ample then $\reg(\iota)=X$
     and \ref{item_lin_span_associated_pts}     is vacuous.
     Similarly,     if $\ccH$ parametrises zero dimensional schemes, then, since $X$ is smooth,
     a general $[R]\in \ccH$ has as the support a collection of general points, which are therefore all in $\reg(\iota)$.
   Thus, in the rest of the proof, we will assume that~\ref{item_lin_span_associated_pts} holds.

   Consider $\ccH^{\circ}$, the open subset as in Notation~\ref{prop_idealmap} and define
   $\spanmap_D:= \perpmap_D \circ \idealmap$.
   Let $[R]\in \ccH$ be a general point. Then
   $[R]\in \ccH^{\circ}$ and all associated points of $R$ are in $\reg(\iota)$.
   We claim that the second assumption guarantees 
   \begin{equation}\label{equ_linear_span_defined_by_ideal}
      \linspan{\iota(R)}
     =\linspan{\overline{\iota(R)}}
     =\PPof{I(R)_D^{\perp}} \subset \PPof{S[X]_D^{*}}.
   \end{equation}
   Indeed, if $\Phi \in I(R)_D$, then $\phi$ considered as a linear form on $\PPof{S[X]_D^{*}}$  must
   vanish identically on $\iota(R)$, hence also on its linear span.
   Conversely, if $R= R_1 \cup \dotsb \cup R_k$ with each $R_i$ a primary scheme,
     then $I(R) =I(R_1)\cap \dotsb \cap I(R_k)$
     and for each $i$, the equality of schemes holds: $R_i=\overline{R_i\cap \reg(\iota)}$.
   Thus, if $\Phi \in S[X]_D$ considered as a linear form on $\PPof{S[X]_D^{*}}$ vanishes on $\iota(R)$,
    it also vanishes on all  $\iota(R_i)$, and thus $\Phi$ vanishes on $R_i\cap \reg(\iota)$
   ---  where $\Phi$ is  considered this time as  an element of $S[X]$.
   Therefore $\Phi \in I(R_i)$ and hence $\Phi \in I(R)$ proving \eqref{equ_linear_span_defined_by_ideal}.

   Thus, with a slight abuse of Notation~\ref{not_parameter_spaces},
   for a general $[R]$ we have
   \[
     \PPof{\spanmap_D([R])}= \PPof{\perpmap_D \circ \idealmap([R])} =
     \PPof{\perpmap_D([I(R)])} = \PPof{I(R)_D^\perp}
   \overset{ \text{by~\eqref{equ_linear_span_defined_by_ideal}}}{=}
     \linspan{\iota(R)},
   \]
   where the first, second and third equalities follow from the definitions of maps
   $\spanmap_D$, $\idealmap$, and $\perpmap_D$ respectively.

   To conclude the proof of the proposition,
   we apply Corollary~\ref{cor_saturation_map_is_birational}.
   The morphism $\sch|_{\mathbf{H}_{\ccH}}$
   is birational and it is the inverse of $\idealmap$,
   hence
   $\spanmap_D\colon \ccH \dashrightarrow \Gr(h_{\ccH}(D), S[X]_D^*)$
      is resolved by the precomposition with
      $\sch|_{\mathbf{H}_{\ccH}}\colon \mathbf{H}_{\ccH} \to \ccH$ as claimed.
\end{prf}

Let $\ccH\subset \usualHilb(X)$ be an irreducible component and fix a very ample line bundle $L\in \Pic(X)$ defining an embedding $\iota\colon X\hookrightarrow \PPof{H^0(L)^*} = \PPof{S[X]_L^*}$ .
For each nonnegative integer $i$
the \emph{Grassmann-relative linear span} of   $\ccH$,
  is defined as
  \[
     \cactus{\ccH,i}{X, L} =
     \overline{\set{
      [E]\mid \PPof{E}\subset \linspan{\iota(R)} \text{ for some } [R]\in \ccH
     }} \subset \Gr(i, S[X]_L^*).
  \]
We will frequently skip $L$ in the notation and write simply $\cactus{\ccH,i}{X}$.
The interesting range of $i$ to consider $\cactus{\ccH,i}{X, L}$ is between $1$ and $h_{\ccH}(L)$,
see Lemma~\ref{lem_interesting_range_for_Grassmann_cactus}.
For the initial case $i=1$ and $\ccH\subset \usualHilb_r(X)$
the relative linear span $\cactus{\ccH,i}{X} \subset \PPof{S[X]_L^*}$
is also defined in \cite[Def.~5.23]{jabu_jelisiejew_finite_schemes_and_secants}
--- see Proposition~\ref{prop_descriptions_of_relative_span} for a comparison of the two definitions.
If $\ccH \subset \usualHilb_r(X)$ is the smoothable component,
  then $\cactus{\ccH,i}{X}$ is the Grassmann-secant variety \cite[Def.~1.1]{landsberg_jabu_ranks_of_tensors}.
More generally, the relative linear span
  of a family of finite schemes (not necessarily flat, and not necessarily over reduced base, but under additional technical assumptions) is defined and discussed in \cite[\S3.2]{jabu_keneshlou_cactus_scheme}.
At the other extreme of the interesting range, for $i=h_{\ccH}(L)$, we have the case that is relevant to the construction of Hilbert scheme.
In fact, if the line bundle $L$ is sufficiently ample, then $\cactus{\ccH,h_{\ccH}(L)}{X} \simeq \ccH$ and the embedding in the Grassmannian is the same as in the construction of the Hilbert scheme.

\begin{lemma}
   \label{lem_interesting_range_for_Grassmann_cactus}
   If $i=0$ then $\cactus{\ccH,i}{X} = \Gr(0, S[X]_L^*) = \Spec \kk$.
   If $i > h_{\ccH}(L)$, then
   $\cactus{\ccH,i}{X} = \emptyset$.
\end{lemma}
\begin{prf}
  If $i=0$, then there is unique $[E]\in \Gr(0, S[X]_L^*)$, namely $E=0$, and $\Gr(0, S[X]_L^*)$ is a single $\kk$-point.
  Since for any $[R] \in \ccH$ the linear span of $R \subset \PPof{S[X]_L^*}$ contains the empty set $\PP(E)$, it follows that $\cactus{\ccH,i}{X} = \Gr(0, S[X]_L^*)$.

  If $i>  h_{\ccH}(L)$, then
    for any $[R]\in \ccH$
    we have $\dim\linspan{R}= h_R(L) -1 \leqslant h_{\ccH}(L)-1$ by Theorem~\ref{thm_structure_of_Hilbert_functions_of_fibres}\ref{item_Hilbert_function_stratification_is_locally_closed} and \ref{item_generic_Hilbert_function_is_open}.
   Since for any $E$ of dimension $i$, $\dim \PP(E)\geqslant h_{\ccH}(L)$, by the dimension count we cannot have $E \subset \linspan{R}$, thus    $\cactus{\ccH,i}{X} = \emptyset$ as claimed.
\end{prf}

\subsection{Grassmannian bundles}
\label{sec_grassmannian_bundles}

In order to discuss several other definitions of relative linear span we first review notation about Grassmannian bundles.
If $Y$ is a scheme, $\ccE$ is a vector bundle on $Y$, and $0\leqslant i\leqslant \rk \ccE$ is an integer,
then by $\relGr(i, \ccE)$
we denote the total space of the Grassmannian bundle:
over $y\in Y$ we consider the fibre $\Gr(i,\ccE_y)$ and we glue these fibres appropriately.
More precisely, we take $\PPof{\Wedge{i}\ccE}$ and define $\relGr(i, \ccE) \subset \PPof{\Wedge{i}\ccE}$ as the closed subscheme using Pl{\"u}cker relations locally with respect to $Y$.
In particular, in the special case $i=1$,
  we have $\relGr(1, \ccE) = \PPof{\ccE}$.

Now consider a vector space $V$ and
  $\ccS \to \Gr(k, V)$ the universal subbundle, that is a subbundle $\ccS $ of the trivial bundle $\Gr(k, V) \times V$
  with $\ccS=\set{([E], v) \mid v \in E}$.
We pick a positive integer $i\leqslant k$.
Then $\relGr(i, \ccS)$ as a variety is in fact isomorphic to the partial flag variety $\Fl(i,k; V)$ and admits two projective morphisms:
\[
\begin{tikzcd}
\relGr(i, \ccS) \arrow[d, "\pi"] \arrow[r, "\xi"] & \Gr(i, V) \\
 \Gr(k, V).
\end{tikzcd}
\]

Now let $X$, $L$, $\ccH\subset \usualHilb(X)$, $\mathbf{H}_{\ccH}$, $\sch\colon \mathbf{H}_{\ccH} \to \ccH$, $\spanmap_L\colon \ccH\to \Gr(h_{\ccH}(L) , S[X]^*_L)$ and
$\perpmap_L\colon \mathbf{H}_{\ccH}\to \Gr(h_{\ccH}(L) , S[X]^*_L)$ be as  in Subsections~\ref{sec_Hilb_schemes} and \ref{sec_linear_span_map}.
We pull back the diagram and obtain:
\begin{equation}
\label{equ_incidence_diagram}
\begin{tikzcd}
&&  & \Gr(i, S[X]^*_L) \\
\relGr(i,\ccS_{\mathbf{H}})\arrow[r, dashed] \arrow[d]\arrow[rrru, bend left=10, "\xi_{\mathbf{H}}"]&\relGr(i,\ccS_{\ccH})\arrow[r] \arrow[d]\arrow[rru, "\xi_{\ccH}"]&  \relGr(i, \ccS) \arrow[d, "\pi"] \arrow[ru, "\xi"'] & \\
\mathbf{H}_{\ccH} \arrow[r, "\sch"]\arrow[rr, bend right=20, "\perpmap_L"']&\ccH\arrow[r, dashed, "\spanmap_L"]& \Gr(h_{\ccH}(L), S[X]^*_L).
\end{tikzcd}
\end{equation}
Here $\ccS_{\ccH}$ is the vector bundle over an open dense subset of $\ccH$ obtained as a pull back of $\ccS$ along $\spanmap_L$,
$ \ccS_{\mathbf{H}}$ is the vector bundle over $\mathbf{H}_{\ccH}$ obtained as a pull back of $\ccS$ along $\perpmap_L$,
and $\relGr(i,\ccS_{\ccH})$, $\relGr(i,\ccS_{\mathbf{H}})$ are the corresponding Grassmannian bundles, while $\xi_{\ccH}$ and $\xi_{\mathbf{H}}$ are the appropriate compositions.

\begin{prop}
   \label{prop_descriptions_of_relative_span}
  Let $\eta\in \ccH$ be the generic point,
   and $\relGr(i,\ccS_{\ccH})_{\eta}$ the generic fibre.
  Then
  \[
   \cactus{\ccH}{X} = \overline{\xi_{\ccH}(\relGr(i,\ccS_{\ccH})_{\eta})} = \xi_{\mathbf{H}} (\relGr(i,\ccS_{\mathbf{H}})).
  \]
\end{prop}

\begin{prf}
  Since $\sch$ is a birational morphism, the generic point of
  $\mathbf{H}_{\ccH}$ is equal to $\sch^{-1}(\eta)$ and it is mapped isomorphically to $\eta$.
  In particular, $\overline{\xi_{\ccH}(\relGr(i,\ccS_{\ccH})_{\eta})}= \overline{\xi_{\mathbf{H}}(\relGr(i,\ccS_{\mathbf{H}})_{\sch^{-1}(\eta)})}$
  and the latter is
  equal to
  $\xi_{\mathbf{H}} (\relGr(i,\ccS_{\mathbf{H}}))$ by Lemma~\ref{lem_flat_family_determined_by_generic_fibre} since $\relGr(i,\ccS_{\mathbf{H}})\to \mathbf{H}_{\ccH}$ is flat.

  Take any $[R]\in \ccH$, and take $[E] \in \Gr(i, V)$ such that $\PP(E)\subset \linspan{R}$.
  Since $\sch$ is surjective,
    there exists $[I]\in \mathbf{H}_{\ccH}$ such that $\sch([I]) = [R]$.
  We have $\linspan{R} \subset \PPof{(I_L)^{\perp}}$,
  thus $E\subset (I_L)^{\perp}$
  and therefore $[E]\in  \xi_{\mathbf{H}} (\relGr(i,\ccS_{\mathbf{H}}))$.
  This proves that $\cactus{\ccH}{X} \subset \xi_{\mathbf{H}} (\relGr(i,\ccS_{\mathbf{H}}))$ since the latter is closed in $\Gr(i, S[X]^*_L)$.

  Finally suppose $[E]\in \xi_{\mathbf{H}} (\relGr(i,\ccS_{\mathbf{H}}))$ is a general point.
  Then there exists $[I]\in \mathbf{H}_{\ccH}$  such that $E\subset I_L^{\perp}$.
  Moreover, since $[E]$ is general,
  by Theorem~\ref{thm_saturation_is_open_property})
  we may assume that $I$ is saturated.
 Then $\sch$ is an isomorphism near $[I]$
    and $I_L^{\perp} = \linspan{R}$, where $[R] = \sch([I])$, see Proposition~\ref{prop_linear_span_map}.
  Thus $E\subset\linspan{R}$ and $[E]\in \cactus{\ccH}{X}$ showing the final inclusion
  $\xi_{\mathbf{H}} (\relGr(i,\ccS_{\mathbf{H}})) \subset \cactus{\ccH}{X}$ and proving the proposition.
\end{prf}

Let us stress that the equality
$\cactus{\ccH}{X, L} = \overline{\xi_{\ccH}(\relGr(i,\ccS_{\ccH})_{\eta})}$ in
Proposition~\ref{prop_descriptions_of_relative_span}, when  $i=1$, is identical to
\cite[Def.~5.23]{jabu_jelisiejew_finite_schemes_and_secants},
which shows compatibility of the two definitions.
Moreover, $\cactus{\ccH}{X} = \xi_{\mathbf{H}} (\relGr(i,\ccS_{\mathbf{H}}))$ does not involve any closure:
in particular, the map $\xi_{\mathbf{H}}\colon \relGr(i,\ccS_{\mathbf{H}})\to \cactus{\ccH}{X}$ is projective and surjective,
which is critical for  applications.

\subsection{Cacti in the literature}
The Grassmann relative linear span introduced in this section is primarily important for
the components $\ccH$ of the Hilbert scheme of points $\usualHilb_r(X)$.
These components are used in Definition~\ref{def_grassmann_cactus_variety} of the Grassmann cactus varieties.
In this subsection we compare our definitions to some of the other occurences in the references.
This is not critical for the rest of the article, but one of the examples serves as a motivation to consider more general $\ccH\subset \usualHilb(X)$ than just components of finite schemes.

\begin{defin}
\label{def_grassmann_cactus_variety}
The \emph{Grassmann cactus variety} $\cactus{r,i}{X}$ is the union of $\cactus{\ccH,i}{X}$ over all components $\ccH$ of $\usualHilb_r(X)$.
\end{defin}

When $i=1$, then $\cactus{r,1}{X}=\cactus{r}{X}$ and it is  called the \emph{cactus variety}.
Cactus varieties are defined as a generalisation of secant varieties.
In \cite[\S5.6--5.7]{jabu_jelisiejew_finite_schemes_and_secants} we define and discuss relative linear spans of components of Hilbert scheme,
and these are used to define both cactus and secant varieties.
On the other hand Grassmann secant varieties are very useful in the context of tensors thanks to
\cite[Cor.~3.6(iii)]{landsberg_jabu_ranks_of_tensors}.

The notion of Grassmann cactus varieties have already been introduced
in \cite{galazka_mandziuk_rupniewski_distinguishing}.
In fact, in their work the authors also use Grassmann-relative linear spans of the unique nonsmoothable component
$\ccH\subset \usualHilb_{8}(\PP^n)$ for $n\geqslant 4$ and of the unique non-smoothable component containing
Gorenstein schemes $\ccH\subset \usualHilb_{14}(\PP^n)$ (for $n\geqslant 6$).
In fact, in their arguments the authors exploit an early draft of this article.

A fundamental property of Grassmann cactus variety is proven in \cite{nisiabu_jabu_galazka_Grassmann_cactus_and_socle}.

As another example,
\cite[Lem~2.4]{ranestad_voisin_VSP_and_divisors_in_the_moduli_of_cubic_fourfolds}
characterises rank $10$ cubics in $6$ variables that are associated to $K3$ surfaces which are complete intersections of
the Grassmannian $\Gr(2, 6) \in \PP^{14}$ with a $\PP^8$.
The characterisation can be reinterpreted as a membership in a relative linear span of Hilbert scheme of quartic surface scrolls in $\PP^5$.

In \cite{landsberg_jabu_ranks_of_tensors} Grassmann secant variety of $X$ is exploited in connection to
secant variety to a product $\PP^{n}\times X$.

\section{Apolarity and border cactus decompositions}

In this section we review basics of multigraded apolarity, define the border cactus decompositions and prove the main results of the article.

\subsection{Apolarity}\label{sec_apolarity}
Let $S[X] = \bigoplus_{D\in \Pic X} H^0(D)$ be the Cox ring of the toric variety in question and let $\widetilde{S}[X] = \bigoplus_{D\in \Pic X} H^0(D)^*$ be the graded dual of $S[X]$.
The apolarity action of $S[X]$ on $\widetilde{S}[X]$ is given by
the contractions $\hook$, as in \cite[\S3.1]{jelisiejew_PhD},
\cite[Equation (1.1) and \S4]{galazka_mgr_publ}, \cite[\S3.1]{nisiabu_jabu_border_apolarity}, and \cite[\S5.1]{jabu_keneshlou_cactus_scheme}.

For a linear subspace $E\subset \widetilde{S}[X]$ its annihilator is the ideal
\[
 \Ann(E) = \set{\Theta \in S[X] \mid \forall_{F\in E} \ \Theta \hook F = 0  } \subset S[X].
\]
If $\dim E= 1$, so that $E$ is spanned by a single $F\in \widetilde{S}[X]$ (or $F\in \PPof{\widetilde{S}[X]_L}$ for some $L$), then we simply write $\Ann(F) = \Ann(E)$.
Note that if $\mathfrak{E}$ is a basis od $E$ as a $\kk$-vector space, then
\[
   \Ann(E) = \bigcap_{F\in E} \Ann(F) = \bigcap_{F\in \mathfrak{E}} \Ann(F).
\]

In this article we focus only on homogeneous case, so that the only interesting case is $E\subset \widetilde{S}[X]_L = H^0(L)^*$ for
a fixed line bundle $L\in \Pic X$. In particular, in this case $\dim E < \infty$, and $\Ann(E)$ is a homogeneous ideal
--- as always ---  with respect to the grading by $\Pic X$.

The standard apolarity type observations are the following lemmas.
\begin{lemma} \label{lem_annihilator_and_perp}
  For any $E\subset H^0(X,L)^*$
    we have $\Ann(E)_{L} = E^{\perp} \subset H^0(X,L)$.
\end{lemma}

\begin{lemma}\label{lem_ideals_contained_in_annihilators}
   Suppose $E\subset H^0(X,L)^*$ is a linear subspace and
   $I \subset S[X]$ is a homogeneous ideal.
    If $I_L \subset \Ann(E)_L$, then $I \subset \Ann(E)$.
\end{lemma}
\begin{prf}
  Note that for any degree $D'\in \Pic X$:
  \begin{equation}\label{equ_if_hook_with_whole_grading_zero_then_G_zero}
     \text{if }
     G\in \widetilde{S}[X]_{D'}
   \text{ is such that } S[X]_{D'}\hook G \equiv 0, \text{ then } G = 0.
  \end{equation}
  Since both $I$ and $\Ann(E)$ are homogeneous ideals,
    it is enough to check the claim in each degree
    $D\in \Pic X $.
  Pick $\Phi\in I_{D}$, and thus $S[X]_{L-D} \cdot \Phi \subset I_{L} \subset \Ann(E)_{L}$ by the assumption of the lemma.
   Hence for all $F\in E$,
   $S[X]_{L-D} \hook (\Phi \hook F) =0$,
   and by \eqref{equ_if_hook_with_whole_grading_zero_then_G_zero}   $\Phi \hook F =0$, that is $\Phi \in \Ann(F)$.
   Intersecting over all $F$, we get the claim:
   $\Phi\in \bigcap_{F\in E} \Ann(F) = \Ann (E)$.
\end{prf}

Although we do not need it explicitly here, the motivation for the border apolarity techniques is the following (standard) apolarity.

\begin{prop}
   \label{prop_multigraded_apolarity}
   Suppose $X$ is a smooth toric projective variety, $L$ is a very ample line bundle, and consider $X$ embedded into $\PPof{\widetilde{S}_L}$.
   Fix a linear subspace $E\subset \widetilde{S}_L$.
   Then for any closed subscheme $R\subset X$, we have the following equivalence:
   \[
     E\subset \linspan{R} \iff I(R)\subset \Ann(E).
   \]
\end{prop}
When the base field is $\kk=\CC$ this is proven in \cite[Thm~1.4]{galazka_mgr_publ}.
For other fields,  but  $X=\PP V$ the projective space this is standard, see for instance
\cite[Lem.~1.15]{iarrobino_kanev_book_Gorenstein_algebras},
and another version of apolarity is \cite[Prop.~4.14]{nisiabu_jabu_farnik_cactus_Fujita}.
In any case, the proof is always the same:
\begin{prf}
  $E\subset \linspan{R}$ if and only if $E\subset (I(R)_L)^{\perp}$ if and only if $I(R)_L \subset E^{\perp}$ if and only if
  (by Lemma~\ref{lem_annihilator_and_perp}) $I(R)_L \subset \Ann(E)_L$ if and only if
  (by Lemma~\ref{lem_ideals_contained_in_annihilators}) $I(R) \subset \Ann(E)$.
\end{prf}

\subsection{Case of general component}
In this subsection $\ccH$ is any irreducible component of $\usualHilb(X)$.
We still use Notation~\ref{not_parameter_spaces}.

\begin{thm}
   \label{thm_border_apolarity_general_scheme}
   Suppose $X$ is a toric projective variety, $L$ is a very ample line bundle, and $\ccH\subset \usualHilb(X)$ is an irreducible component.
   Denote by $\mathbf{H}_{\ccH}\subset \Hilb(S[X])$
   the corresponding  component of the multigraded scheme,
   as in Corollary~\ref{cor_saturation_map_is_birational}.
   Consider $X$ embedded into $\PPof{\widetilde{S}_L}$
   and fix a point $[E] \in \Gr\left(i, \widetilde{S}_L\right)$.
   Then:
   \[
     [E]\in \cactus{\ccH,i}{X} \iff
     \exists_{[I] \in \mathbf{H}_{\ccH}} \text{ s.t. } I \subset \Ann(E).
   \]
\end{thm}

\begin{prf}
   Recall from Subsection~\ref{sec_grassmannian_bundles} the construction of bundles $\ccS_{\mathbf{H}}$  and $\Gr(i, \ccS_{\mathbf{H}})$.
   In particular, the $\kk$-points of $\Gr(i, \ccS_{\mathbf{H}})$ are precisely
   $\set{([I], [E']) \in \mathbf{H}_{\ccH} \times \Gr(i,\widetilde{S}_L)  \mid  E' \subset I_L^{\perp}}$.
   Equivalently,
     by Lemmas~\ref{lem_annihilator_and_perp}
     and~\ref{lem_ideals_contained_in_annihilators},
   \[
     \Gr(i, \ccS_{\mathbf{H}})(\kk)=\set{([I], [E']) \in \mathbf{H}_{\ccH} \times \Gr(i,\widetilde{S}_L)  \mid  I \subset \Ann(E')}.
   \]
   By Proposition~\ref{prop_descriptions_of_relative_span}
   $[E]\in \cactus{\ccH,i}{X}$  if and only if
      there exists $[I]\in \mathbf{H}_{\ccH}(\kk)$ such that $([I],[E]) \in \Gr(i, \ccS_{\mathbf{H}})$.
   The condition on $([I],[E])$ is then equivalent to $I \subset \Ann(E)$, proving the Theorem.
\end{prf}

In the setting of Theorem~\ref{thm_border_apolarity_general_scheme}
we define the \emph{variety of border cactus decompositions}
$\bCD(E, \ccH)$ to be
the reduced fibre over $[E]$ of the map
$\xi_{\mathbf{H}} \colon \relGr(i,\ccS_{\mathbf{H}}) \to \Gr(i, \widetilde{S}_L)$
as in Subsection~\ref{sec_grassmannian_bundles},
$\bCD(E, \ccH):=\reduced{\xi_{\mathbf{H}}^{-1} (E)} \subset \mathbf{H}_{\ccH}$.
Note that $\bCD(E, \ccH)$ is a projective reduced scheme, and it is non-empty if and only if $[E]\in \cactus{\ccH,i}{X}$.

Suppose $G\subset \Aut(X)$ is a connected subgroup of the automorphisms of $X$.
Then it has natural induced actions on $\ccH$, $\mathbf{H}_{\ccH}$, and other varieties and bundles appearing in Diagram~\eqref{equ_incidence_diagram}, such that all the maps in the diagram are equivariant.
If in addition $G$ preserves $[E]\in \Gr(i, \widetilde{S}_L)$ then it also acts on $\bCD(E, \ccH)$.

\begin{thm}
   \label{thm_FIT_for_general_scheme}
   With $\ccH$, $E$, and $G$ as above with $G\cdot E = E$,
   assume in addition that $G$ is a solvable group.
   Then $E\in \cactus{\ccH,i}{X}$  if and only if there exists an ideal $[I]\in \mathbf{H}_{\ccH}$
   preserved by $G$ and such that $I \subset \Ann(E)$.
\end{thm}
\begin{prf}
   By the above discussion $[E]\in \cactus{\ccH,i}{X}$ if and only if $\bCD(E, \ccH)$ is non-empty.
   By Theorem~\ref{thm_border_apolarity_general_scheme}
    we have
    \[
       \bCD(E, \ccH)(\kk) = \set{[I]\in \mathbf{H}_{\ccH}
       \mid I \subset \Ann(E)}.
    \]
    Since $\bCD(E, \ccH)$ is a projective scheme, by Borel's fixed point lemma \cite[Thm~10.4]{borel} it is non-empty if and only if it contains a $G$-fixed
    $\kk$-point.
    A $G$-fixed point in $\bCD(E, \ccH)$
    is a $G$-invariant ideal $I\subset S[X]$ contained in $\Ann(E)$, proving the claim.
\end{prf}

For the rest of the paper we restrict our attention to components $\ccH\subset \usualHilb_r(X)$, that is, we restrict
to the Hilbert scheme of points.

\subsection{Case of secant varieties}

In this subsection, we let $\ccH = \usualHilbsm{X}$, the smoothable component of the Hilbert scheme.
In this case the standard terminology and notation is the following:
\begin{itemize}
 \item $\cactus{\ccH,i}{X} = \sigma_{r,i}(X)$ is the Grassmann secant variety, or simply secant variety
 $\sigma_r(X)$ for $i=1$,
 \item $\mathbf{H}_{\ccH} =\Slip_r(X)$ is the scheme of limits of ideals of points,
 \item the Hilbert function $h_{\ccH}(D)$ is equal to $h_{r,X}:= \min(r, \dim_{\kk} S[X]_D)$ by Theorem~\ref{thm_Sip_open}.
\end{itemize}

A special feature of this case is the relation between Grassmann secant variety of $X$ and secant variety of $\PP^{i-1}\times X$, see for instance \cite[Thm~2.5, Cor.~3.6(iii)]{landsberg_jabu_ranks_of_tensors}.

Thus for the case $\ccH = \usualHilbsm{X}$ we obtain:
\begin{thm}
   \label{thm_border_apolarity_smoothable_and_slip}
   Suppose $X$ is a smooth toric projective variety, $L$ is a very ample line bundle.
   Consider $X$ embedded into $\PPof{\widetilde{S}_L}$
   and fix a linear subspace $E\subset \widetilde{S}_L$ of dimension~$i$.
   Then:
   \[
     [E]\in \sigma_{r,i}(X) \iff
     \exists_{[I] \in \Slip_r(X)} \text{ such that } I \subset \Ann(E).
   \]
   Moreover, if $G\subset \Aut(X)$ is a connected solvable subgroup that preserves $E$,
    then
    \[
      [E]\in \sigma_{r,i}(X) \iff
       \exists_{[I] \in \Slip_r(X)} \text{ such that } I \subset \Ann(E)\text{ and } G\cdot I =I.
    \]
\end{thm}
The theorem follows directly from Theorems~\ref{thm_border_apolarity_general_scheme} and \ref{thm_FIT_for_general_scheme}.
Moreover, restricting to $i=1$,
Theorems~\ref{thm_Sip_open}
and \ref{thm_border_apolarity_smoothable_and_slip}
imply Theorem~\ref{thm_border_apolarity_intro}.

As an interesting conclusion, if the border rank of a monomial can be bounded from below using weak border apolarity (Theorem~\ref{thm_border_apolarity_intro}) for some base field, then the same bound applies over any base field.
We illustrate this on a simple example without providing details.
\begin{example}
   Let $X=\PP^2$ and $m = x_0^{(a_0)}x_1^{(a_1)}x_2^{(a_2)}$ with $a_0\leqslant a_1 \leqslant a_2$.
   Then $m \in \sigma_r(X) \setminus \sigma_{r-1}(X)$
   for $r=(a_0+1)(a_1+1)$,
   independent of the base field $\kk$.
   Indeed, the upper bound on such $r$ is proven in the same way as \cite[Lem.~6.2]{nisiabu_jabu_border_apolarity}.
   The lower bound is proven in the same way as \cite[Thm~6.21]{nisiabu_jabu_border_apolarity},
     which boils down to prove non-existence of a monomial ideal $I$ with a fixed Hilbert function contained in another fixed monomial ideal.
   This is purely combinatorial,
      and does not depend on the base field $\kk$.
\end{example}

Up to now, all the monomials that have known border rank over $\CC$, can be calculated using weak border apolarity.

\subsection{Case of cactus varieties}

Now we compare the standard terminology for cactus varieties with the relative linear spans.
Explicitly, let $\usualHilb_{r}(X)$ be the Hilbert scheme of $r$-points, which might have multiple components.
The $r$-th \emph{Grassmann cactus variety} of $X$ is defined as:
\[
  \cactus{r, i}{X}:= \bigcup_{\ccH \subset \usualHilb_{r}(X)} \cactus{\ccH, i}{X}  \subset Gr(i, \widetilde{S}_L)
\]
where the union is over all irreducible components $\ccH$ of $\usualHilb_{r}(X)$.
As before, the special case $i=1$ is simply called the \emph{cactus variety} of $X$ and denoted:
\begin{equation}
  \label{equ_defin_cactus}
  \cactus{r}{X}:= \bigcup_{\ccH \subset \usualHilb_{r}(X)} \cactus{\ccH}{X}  \subset \PPof{\widetilde{S}_L}.
\end{equation}

\begin{prf}[ of Theorem~\ref{thm_abcd}]
   With the assumptions of the theorem we work in the case
   $i=1$.
   By \cite[Cor.~6.20]{jabu_jelisiejew_finite_schemes_and_secants}
    in Equation~\ref{equ_equality_for_generic_fibre_algebraically} it is enough to take the union only over those components $\ccH \subset \usualHilb_{r}(X)$ such that its general element represents a Gorenstein scheme.
    The rest of the claim follows from Theorem~\ref{thm_border_apolarity_general_scheme}.
\end{prf}

Now Theorem~\ref{thm_weak_abcd} follows immediately, as there are only finitely many components of the Hilbert scheme, each of them has one Hilbert function of a general member by Theorem~\ref{thm_structure_of_Hilbert_functions_of_fibres}\ref{item_generic_Hilbert_function_is_open}.
Note that different components may have the same Hilbert function.
For the potential applications, it is an important problem to list all these Hilbert functions, or at least discover the properties of these Hilbert functions.
Naively, one might hope that the only relevant function is $h_{r,X}:= \min(r, \dim_{\kk} S[X]_D)$,
however, see Proposition~\ref{prop_example_of_nongeneral_Hilbert_function}.
At this point we do not know enough about components of Hilbert schemes of points to be able to determine the Hilbert function of a general scheme in each component (except in the situation when the resulting Hilbert function is equal or very close to $h_{r,X}$).

More generally, properties of Hilbert functions of saturated ideals form an interesting topic,
see \cite{ablett_Gotzmann_persistence_for_smooth_projective_toric_varieties} and references therein.
All these properties apply to $h_{\ccH}$.
The only new information we are able to provide is the following observation.

\begin{prop}\label{prop_genhilbfunc}
   Suppose $\ccH\subset \usualHilb_r(\PP^n)$ is an irreducible component of the
      Hilbert scheme of finite subschemes of length $r$.
   Then $h_{\ccH}(1) = \min\set{r, n+1} = h_r(1)$.
\end{prop}

\begin{lemma}\label{lem_nondeg_embedding}
   Suppose $i_1\colon R \hookrightarrow \PP V$ is an embedding of a finite scheme of length $r$.
   Then there is an embedding $i_2\colon R \hookrightarrow \PP V$
     such that $\dim (V^*/I(i_2(R))_1) = \min\set{r, \dim V}$.
\end{lemma}

\begin{proof}
   First, there is always a concisely independent embedding
   $i_3\colon R \hookrightarrow \PP^{r-1}$
   with $I(i_3(R))_1 = 0$
   \cite[pp~702--703]{nisiabu_jabu_kleppe_teitler_direct_sums}.
   This is obtained by the embedding $R = \Spec A$ into $\PP A$,
     or equivalently, by the trivial line bundle $\ccO_R$, which is very ample.
   If $\dim V\ge r$, then $i_2$ is a composition of $i_3$ with a linear embedding $\PP^{r-1} \subset \PP V$.

   Suppose $\dim V < r$.
   Let $\cactus{2}{i_3(R)}\subset \PP^{r-1}$ be the second cactus variety of $R$,
   that is the finite union of the (projective) Zariski tangent spaces of $R$ at each point and the secant lines connecting the any two points of support of $R$.
   Note that $\dim \cactus{2}{i_3(R)} \le \dim \PP V$:
   indeed, since $i_1$ is an embedding of $R$ into $\PP V$,
   each tangent space must be at most $\dim \PP V$ dimensional.
   Also each secant line is one dimensional, and there are any secant lines only if there are at least two distinct points of support,
   which is possible only if $\dim \PP V\ge 1$.

   Pick a linear projection $\PP^{r-1}\dashrightarrow \PP V$.
   By standard arguments, such as in \cite[Prop.~IV.3.4]{hartshorne},
      if the center of the projection does not intersect
      $\cactus{2}{i_3(R)}$,
      then the composition of $i_3$ and the projection is still an embedding.
   Moreover, in such a case the linear span of this new embedding is $\PP V$.
   Clearly, by the dimension count, we can pick the linear projection satisfying the above property.
\end{proof}

\begin{lemma}\label{lem_two_embeddings_in_the_same_component}
   Suppose $i_1$ and $i_2$ are two embeddings of the same finite scheme
   $R\hookrightarrow \PP V$.
   Then $[i_1(R)]$ and $[i_2(R)]$ are in the same irreducible component of the Hilbert scheme $\usualHilb_r(\PP V)$, where $r=\deg R$.
\end{lemma}
\begin{prf}
   Pick a basis $\setfromto{\alpha_0}{\alpha_n}$
      of $V^* = H^0(\ccO_{\PP V} (1))$.
   Pulling the basis back to $R$ using $i_1$ and $i_2$,
      we obtain two collections,
      $\setfromto{i_1^*\alpha_0}{i_1^*\alpha_n}$ and
      $\setfromto{i_2^*\alpha_0}{i_2^*\alpha_n}$,
      of sections of the trivial line bundle
      $\ccO_R \simeq i_1^*\ccO_{\PP V} (1) \simeq i_2^*\ccO_{\PP V} (1)$.
   Consider the $\PP^1$-parametrised family of maps $R \to \PP V$
      determined by
      \[
         \setfromto{s \cdot i_1^*\alpha_0+ t\cdot i_2^*\alpha_0}{s \cdot i_1^*\alpha_n+ t\cdot i_2^*\alpha_n},
      \]
      where $s,t$ are coordinated on $\PP^1$.
   Generically, this is an embedding, and thus we obtain a flat family of subschemes of $\PP V$ paramerised by an open dense subset of $\PP^1$.
   This family demonstrates that $[i_1(R)]$ and $[i_2(R)]$ are in the same component of the Hilbert scheme, as claimed.
\end{prf}

\begin{prf}[ of Proposition~\ref{prop_genhilbfunc}]
   Take a general element $[R]\in \ccH$. By the generality, $\ccH$ is the unique component of $\usualHilb_r(\PP V)$ containing $[R]$.
   By Lemma~\ref{lem_nondeg_embedding}, we can reembed $R$ into $\PP V$
      in a linearly nondegererate way.
   By Lemma~\ref{lem_two_embeddings_in_the_same_component}
      both embeddings are in the same component, that is $\ccH$.
   It follows that $h_{\ccH}(1)\ge \min\set{r, \dim V}$,
      then we must have the equality since $h_{\ccH}(1)\leqslant h_r(1)$.
\end{prf}

Consider a set of $4$ points in $\PP^n$.
If three of them are on a line, then all four of them are linearly dependent, contained in a $\PP^2$.
The following lemma generalises this elementary observation to schemes, and we skip the proof.
\begin{lemma}
   \label{lem_defective_subscheme}
   Suppose $R$ and $R'$ are two finite schemes of degree $r$ and $r'$ respectively.
   Suppose $R'\subset R \subset X$.
   Then for any $D\in \Pic X$ we have the inequality
   $h_R(D) \leqslant h_{R'}(D) + (r - r')$.
   In particular, if $h_{R'}(D) < r'$, then also $h_{R}(D) < r$.
\end{lemma}
\noprf

\begin{prop}
   \label{prop_example_of_nongeneral_Hilbert_function}
   Suppose $\kk=\CC$ and consider $\ccH \subset \usualHilb_{28}(\PP^6)$, the component, whose general element $[R]\in \ccH$ represents a scheme $R\subset \PP^6$ supported at two points and locally, near each point $R$ is isomorphic to a $(1661)$ Gorenstein scheme, as in \cite[Thm~1.1]{jelisiejew_VSP_and_Gorenstein_locus_for_14_points}.
   Then
   \[
      h_{R}(2) = h_{\ccH}(2) \leqslant 27 < h_{28, \PP^6}(2) = 28.
   \]
\end{prop}

\begin{prf}
   For $i=1, 2$, let $R_i\subset R$ be the irreducible components of $R$, so that each $R_i$ is a $(1661)$ Gorenstein local scheme.
   Let $R'_i \subset R_i$ be the length $7$ subscheme defined by the square of the maximal ideal and set $R'=R'_1 \sqcup R'_2$.
   Then by an explicit computation (or, if we want to use the cannon to kill a fly, by Alexander-Hirschowitz Theorem \cite[Thm~1.1]{brambilla_ottaviani_on_AH_theorem}),
   we have $h_{R'}(2) = 13$.
   The degree of $R$ is $28$, while the degree of $R'$ is $14$.
   By Lemma~\ref{lem_defective_subscheme} we must have:
   \[
      h_{\ccH}(2) = h_{R}(2)  \leqslant h_{R'}(2) + (28-14) = 27.
   \]
   This concludes the proof.
\end{prf}

By an explicit computer calculation one may show that for $\ccH$ as in Proposition~\ref{prop_example_of_nongeneral_Hilbert_function}, in fact, $h_{\ccH} = (1,7, 27,28,28,\dotsc)$,
hence $h_{\ccH}(d) = h_{28, \PP^6}(d)$ for all $d\neq 2$,
while $h_{\ccH}(2) =27 < h_{28, \PP^6}(2)= 28$.

%

\bibliography{abcd.bbl}

\end{document}